\documentclass{amsart}
 
\usepackage{amssymb}
\allowdisplaybreaks
\usepackage{stmaryrd} 
 
\usepackage{graphicx}
\usepackage[caption=false]{subfig}
 
\usepackage{float}
\usepackage[ruled,vlined,linesnumbered]{algorithm2e}
\SetKwInput{Input}{Input}
\SetKwInput{Output}{Output}
 
\usepackage{xcolor}
\usepackage{hyperref}
\hypersetup{
  bookmarksopen=true,
  bookmarksopenlevel=2,
  colorlinks=true,
  urlcolor=blue,
  citecolor=black!70!green,
  linkcolor=black!70!red
}

\usepackage{multirow}
\usepackage{booktabs}
 
\usepackage{tikz}
\usetikzlibrary{matrix,arrows, positioning}
\usetikzlibrary[decorations.pathmorphing,mindmap]
\usetikzlibrary{mindmap,matrix,arrows,decorations.pathmorphing,patterns,fadings}
\usetikzlibrary{cd} 
 
\usepackage{pgfplots}
\pgfplotsset{compat=1.18}
\usepackage{pgfplotstable}
\usepackage{enumitem}
 
\usepackage{mathtools}   
\usepackage{xurl}        
\usepackage{microtype}   
\usepackage{lmodern}     
 
\theoremstyle{definition}
\newtheorem{definition}{Definition}
 
\theoremstyle{plain}
\newtheorem{theorem}{Theorem}
\newtheorem{proposition}{Proposition}

\newtheorem{lemma}{Lemma}
\newtheorem{corollary}{Corollary}
 
\theoremstyle{remark}
\newtheorem{remark}{Remark}
 
\theoremstyle{plain}

\newcommand{\legendre}[2]{\ensuremath{\left( \frac{#1}{#2} \right) }}
 
\DeclareMathOperator{\GL}{GL}

\DeclareMathOperator{\Div}{Div}
\DeclareMathOperator{\Pic}{Pic}

\DeclareMathOperator{\NS}{NS}
\DeclareMathOperator{\rank}{rank}
\DeclareMathOperator{\Hom}{Hom}
\DeclareMathOperator{\End}{End}
\DeclareMathOperator{\charec}{char}

\DeclareMathOperator{\cl}{cl}

\DeclareMathOperator{\red}{red}
\DeclareMathOperator{\tr}{tr}
\DeclareMathOperator{\trd}{trd}
\DeclareMathOperator{\nrd}{nrd}

\DeclareMathOperator{\Gen}{Gen}

\newcommand{\F}{\mathbb{F}}
\newcommand{\Q}{\mathbb{Q}}
\newcommand{\Z}{\mathbb{Z}}

\newcommand{\curlyO}{\mathcal{O}}
\newcommand{\A}{\mathcal{A}}
\newcommand{\mi}{\mathbf{i}}
\newcommand{\mj}{\mathbf{j}}
\newcommand{\mL}{\mathcal{L}}
 
\title{Refined Humbert Invariants in Supersingular Isogeny Degree Analysis}
 
\author{Eda K{\i}r{\i}ml{\i}}
\address{University of Birmingham, United Kingdom}
\email{e.kirimli@bham.ac.uk}
 
\author{Gaurish Korpal}
\address{University of Auckland, New Zealand}
\email{gaurish.korpal@auckland.ac.nz}
 
 
\begin{document}

\begin{abstract}
  We focus on \emph{refined Humbert invariants} of principally polarized superspecial abelian surfaces, introduced by Kani in 1994. The main contributions are to enumerate principal polarizations on a superspecial surface, and for each polarization, to compute the refined Humbert invariant of a principally polarized superspecial abelian surface. Then, we present several applications of computing this invariant for isogeny-based cryptography. 
  First, we provide a decision algorithm to check if two given polarizations are isomorphic.
  Second, we present an efficient algorithm to determine the geometric type of a principally polarized superspecial surface.
  Third, we prove an upper bound on the largest minimal isogeny degree among pairs of supersingular elliptic curves, independent of their endomorphism-ring structures, and our experimental evidence verifies this claim up to $p=659$, $p\equiv 11\pmod{12}$.
  Fourth, we present experimental evidence for a minimum isogeny frequency within the proven upper bounds.
   Lastly, we provide a different perspective on the fixed isogeny degree problem using refined Humbert invariants and analyze it without explicit endomorphism rings. 
  
\end{abstract}

\maketitle

\section{Introduction}
The refined Humbert invariants of principally polarized superspecial surfaces underline several hard isogeny problems relevant to post-quantum cryptography~\cite{edachloe}, yet no prior work has computed them for superspecial surfaces or enumerated their principal polarizations. In this paper, we close this gap: we give the first computation of refined Humbert invariants of principally polarized superspecial surfaces, the first enumeration of their principal polarizations, and apply the resulting invariants to several isogeny problems.

Our strategy is based on refined Humbert invariants, whose definition and relevant properties we briefly summarize below.
 Let $\A$ be an abelian surface over $\overline{\F}_{p}$ and let $\NS(\A)$ denote its N\'eron-Severi group. The intersection product $(D_1\cdot D_2)$ of divisors $D_1$ and $D_2$ on $\A$ defines an integral quadratic form $q_{\A}$ on $\NS(\A)$,  Definition~\ref{intersectionform}. 
Since $\NS(\A)\cong \Z^{\rho}$ where $\rho = \rho(\A)$ is the Picard number of $\A$, the quadratic form\footnote{Throughout the paper, a quadratic form means always an integral quadratic form.} $q_{\A}$ is equivalent to an integral quadratic form $q$ in $\rho$ variables, so leading to an isomorphism $(\NS(\A), q_{\A})\cong (\Z^{\rho}, q)$ of quadratic modules.
 
Let $\theta$ be a principal polarization on $\A$, see Definition \ref{polarizations}. We define an integral quadratic form $\tilde{q}_{(\A,\theta)}$ on $\NS(\A)$ as
\begin{equation*}
\tilde{q}_{(\A,\theta)}(D) = (D \cdot \theta)^2 - 2(D \cdot D) \textrm{ for } D \in \NS(\A).
\end{equation*}
It follows that $\tilde{q}_{(\A,\theta)}(D + n\theta) = \tilde{q}_{(\A,\theta)}(D)$ for all $n\in \Z$, therefore we have a quadratic form $q_{(\A,\theta)}$ defined on the quotient module $\NS(\A,\theta) = \NS(\A)/ \Z \theta$, called the polarized  N\'eron-Severi group. The form $q_{(\A,\theta)}$ is a positive-definite quadratic form on the polarized N\'eron-Severi group $\NS(\A,\theta)\cong \Z^{\rho-1}$, and called a \emph{refined Humbert invariant}, see Section~\ref{sec:RHI} for details. 
 
 The computation of refined Humbert invariants is a challenging problem; this has been thoroughly examined in \cite{Kani94,KaniJac14,KaniMJ,KaniESCII,KaniESCI,KaniSubcoversofCurves}. The determination of a refined Humbert invariant is comparatively easier when $\A\cong E_1 \times E_2$ for elliptic curves $E_1$, $E_2$. In this case, there exists a relation between divisors in the N\'eron-Severi group and isogenies between $E_1$ and $E_2$. By using Theorem~\ref{mainisomorphism}, we can represent divisors on $\NS(E_1\times E_2)$  as $D=\mathcal{D}(a,b,\varphi)$ with two integers $a,b$ and an isogeny $\varphi$. This representation allows us to compute the intersection number of divisors more easily. The relation between the refined Humbert invariant and the degree map is given by Lemma \ref{fromRHItodegreemap}
\begin{equation*}
q_{(E_1 \times E_2, \theta_{E_1\times E_2})} \textrm{ is $\Z$-equivalent to } x^2 + 4 q_{E_1,E_2} 
\end{equation*}
where $\theta_{E_1\times E_2}$ is a product polarization and $q_{E_1,E_2} \in \Z[y_1,\ldots,y_n]$ is the degree map.
 
First, we develop a decision algorithm to determine whether two principal polarizations on a superspecial surface are isomorphic. Although this problem has a classical formulation as the equivalence problem for positive-definite binary quaternion Hermitian forms \cite{ibukiyama1986supersingular, HashIbu}, i.e., for a maximal quaternion order $\curlyO$, testing equivalence of principal polarizations given by two uni-modular Hermitian matrices $g_1, g_2\in \GL_{2}(\mathcal{O})$ requires finding a matrix
$u\in \GL_{2}(\mathcal{O})$
satisfying $g_{2}=u^{*}g_{1}u$ where $*$ is the conjugate-transpose. This problem is computationally demanding, and, to the best of our knowledge, no decision algorithm has previously been given for the principal polarizations of superspecial abelian surfaces.
 
Secondly, we provide an explicit algorithm to determine the geometric type of a principally polarized superspecial abelian surface: whether it is the Jacobian of a genus-two curve with its canonical principal polarization or a product of two elliptic curves with the product polarization. This is achieved by the criterion in Proposition~\ref{irreducibilitycriterion}.
 
Thirdly, we prove that, for a generic prime $p$, the maximum of the minimum distances between supersingular elliptic curves is upper-bounded by $\sqrt{\frac{p}{2}}$, without prior explicit knowledge of their endomorphism rings.  
We first discuss previous proven bounds, improve the bound itself, and validate their correctness with experiments for primes $p< 660$, $p\equiv 11\pmod {12}$, by assuming less information. Without refined Humbert invariants, verifying this heuristic for small examples would involve brute-force computations of isogenies or quaternion orders between pairs of elliptic curves. Our approach instead reduces the problem to simple checks on quadratic forms.
 
Fourthly, we run experiments on the frequency and distribution of isogeny degrees between supersingular elliptic curves defined over $\mathbb{F}_{p^2}$. We run these experiments directly on the degree maps themselves, without computing isogenies or endomorphism rings. 
 
Lastly, we give a different approach to the fixed-degree isogeny problem: given supersingular curves $E_1,E_2$ and a target degree $N$, we show that an efficient algorithm for computing refined Humbert invariants yields the degree map $q_{E_1,E_2}$, addressing the intermediate-degree range left open between lattice-reduction methods for small degrees and KLPT-type algorithms~\cite{KLPT,SQIsign} for large ones. This bears directly on schemes such as SQIsign, whose verification step is non-deterministic because \textsf{SigningKLPT} does not guarantee a specific isogeny degree, slowing the protocol.
 
\paragraph{Contribution.}
We start with a supersingular product surface $\A = E_1\times E_2$, and choose a principal polarization $\theta$ and an arbitrary divisor $D$ on  $\NS(E_1\times E_2)$ to calculate the refined Humbert invariant, see Definition~\ref{RHIdefn}. Thus, we apply the irreducibility criterion, Proposition~\ref{irreducibilitycriterion}, of the refined Humbert invariant $q_{(\A,\theta)}$  to decide whether it is a product of two elliptic curves or not. Later, we use the relation between the refined Humbert invariant and the degree map (Definition \ref{RHI=x2+4}), then pull out the degree map and check its minimum value. This same degree map underlies our experiments on isogeny-degree distributions and our treatment of the fixed-degree isogeny problem.
 
In terms of computational assumptions on the degrees of isogenies between supersingular elliptic curves,  we want to investigate the minimal integer represented by such degree maps for different choices of principal polarizations $\theta$, as they are exactly the minimum isogenies between two supersingular elliptic curves. Then, we aim to find the largest possible minimal degree that allows us to bound the maximum minimum isogeny degree reached by every pair of supersingular elliptic curves. That is, for a fixed prime $p$, we aim to understand the value 
\begin{equation}\label{maxmin}  
d\coloneqq \max_{E_1,E_2} \{ \min \{N :q_{E_1,E_2}(t_1,t_2,t_3,t_4) = N \text{ for some } t_1, t_2,t_3,t_4 \in \mathbb{Z}\}\}.\end{equation}
This set ranges over different supersingular curves $E_1$ and $E_2$ over $\F_{p^2}$.
 
\paragraph{Outline.}
The paper is organized into preliminaries, computations, and applications. In Section~\ref{sec:prelim}, we summarize the theoretical preliminaries. In Section~\ref{designing}, we develop algorithms to compute refined Humbert invariants and degree maps, and present how computations were performed. In Section~\ref{sec:apps}, we discuss the applications of our method, including the proof of the upper bound on the minimum isogeny degrees. We conclude the paper in Section~\ref{sec:conclusion}.

\section{Preliminaries}  \label{sec:prelim}
\subsection{Refined Humbert invariant} \label{sec:RHI}

In this section, we introduce the theory of refined Humbert invariants defined by Kani \cite{Kani94} in 1994, which is the main tool of this paper. This invariant is very beneficial in the interplay between geometric and arithmetic problems. Many applications of these invariants can be found in \cite{KaniJac14,KaniMJ}. This section mainly follows \cite{KaniESCII,KaniESCI}. We only present essential facts of refined Humbert invariants here. 

\begin{definition}\cite[p.357]{hartshorne2013algebraic}\label{divisorequivalence}
Let $\A$ be an abelian surface over a field $K$. 
Let $\Div(\A)$ be the set of divisors of $\A$.
If $D_1$ and $D_2 \in \Div(\A)$, then we say that $D_1$ is \emph{numerically equivalent} to $D_2$, denoted by $D_1 \equiv D_2$, and if for all $D \in \Div(\A)$ we have that
\[(D_1 \cdot D) = (D_2 \cdot D),\]
where $(\cdot)$ denotes the intersection number.
\end{definition}

The intersection theory of abelian varieties is a vast subject; for technical details, we refer to \cite[Chapter 4.1]{shafarevich} and  \cite[Chapter V.1]{hartshorne2013algebraic}. Fortunately, we utilize a simple formula for the intersection formula in the case of $\A=E_1\times E_2$ given by Theorem~\ref{mainisomorphism} below, and it is sufficient for our calculations.

\subsubsection{N\'eron-Severi group.} Let $\A/K$ be an abelian surface. We define the \emph{N\'eron-Severi group} $\NS(\A)$ of $\A$ to be 
    \[\NS(\A) = \Div(\A)/\equiv,\]
where $\equiv$ is the equivalence defined in Definition~\ref{divisorequivalence}.
\noindent
This definition agrees with the usual definition of $\NS(\A)$ in \cite[p.101]{langabelianvar} since the $\Pic^0(\A)$-equivalence and the numerical equivalence coincide by the Corollary of Theorem V.1 in \cite{langabelianvar}.

\begin{definition}\label{polarizations}
Let $\A/K$ be an abelian surface. We define
\begin{align*}
\mathcal{P}(\A)= \{\cl(D)\in \NS(\A) :D\in \Div(\A) \emph{ is ample and } (D\cdot D) =2 \}.
\end{align*}
to be \emph{the set of principal polarizations} of $\A$.
\end{definition}
\noindent
By using the Nakai-Moishezon criterion~\cite[Theorem V.1.10]{hartshorne2013algebraic}, we see that if $D \in \NS(\A)$ is ample, then $D'\in \cl(D)$ is ample. 

\begin{definition}\label{intersectionform}
The intersection product $(D_1 \cdot  D_2)$ of divisors $D_1$, $D_2$ on an abelian surface, $\A$ defines an integral quadratic form $q_{\A}$ on $\NS(\A)$, called the \emph{intersection form}:$$q_{\A}(D)=\frac{1}{2}(D\cdot D) \emph{ for all } D\in \Div(\A).$$ 
\end{definition}
\noindent
Since $\NS(\A)\cong \Z^{\rho}$ where $\rho = \rho(\A)$ is the Picard number of $\A$, the form $q_{\A}$ is equivalent to an integral quadratic form $q$ in $\rho$ variables, so we obtain an isomorphism $(\NS(\A), q_{\A})\cong (\Z^{\rho}, q)$ of quadratic modules.

\begin{definition} \cite[Corollary III.6.3]{silverman2009arithmetic}\label{defn:degreemap}
We define the \emph{degree map} (or \emph{degree quadratic form})
$$ q_{E_1,E_2}(\varphi) = \deg(\varphi) \emph{  for  }  
\varphi\in \Hom(E_1,E_2).$$
The degree map $q_{E_1,E_2}$ is a positive definite quadratic form on $\Hom(E_1, E_2)$ in $r$ variables, where $r = \rank(\Hom(E_1, E_2)) = \dim_{\mathbb{Q}}(\End^{0}(E_i))$.\end{definition}
\noindent
For a fixed basis of $\Hom(E_1, E_2)$,  the degree map $q_{E_1, E_2}$  is an explicit positive definite quadratic form in $r$ variables.

\begin{definition} \cite[Corollary of Section 6.4]{fultonalgebraiccurves}

Let $X$ and $Y$ be varieties. If $h :X \rightarrow Y$ is a morphism of varieties, the graph of $h$, denoted by $\Gamma_h$ is defined to be
$\{(x, y) \in X \times Y | \quad y = h (x)\}$.
\end{definition}

We now specialize in the case of products of two elliptic curves.
\begin{theorem} \emph{\cite[Proposition 22]{KaniMJ}}\label{mainisomorphism} 
Let $\A= E_1 \times E_2$ be a product of two elliptic curves. Then we have a group isomorphism 
\begin{align} 
   \mathcal{D}:= \mathcal{D}_{E_1,E_2} :\mathbb{Z} \oplus \mathbb{Z} \oplus \Hom(E_1,E_2) \longrightarrow \NS(\A) \\\empty
    \mathcal{D}(a, b, \varphi) = (a-1)\theta_1+(b-\deg (\varphi))\theta_2 + \Gamma_{-\varphi}
\end{align}
where $\theta_i = p_i^{*}(0_{E_i})$, and $\Gamma_{-\varphi}$ is the graph of $-\varphi$. Then the rule $(a, b, \varphi)\rightarrow  \mathcal{D}(a, b, \varphi) 
\in \NS(\A)$ defines a group
isomorphism Moreover, for two divisors $D_1= \mathcal{D}(a,b,\varphi)$ and $D_2= \mathcal{D}(a',b',\varphi')$ in $\NS(\A)$, the intersection number of the divisors is given by 
\begin{equation}\label{intersectionoftwodivisors}
    (D_1 \cdot D_2) = ab'+ a'b - \beta_d(\varphi,\varphi')
\end{equation}
where $\beta_d$ is the bilinear form associated with the $q_{E_1,E_2}$ on $\Hom(E_1,E_2)$. Thus,\footnote{We would like to emphasize that there is no isogeny computation needed in the intersection formulas, and only the degrees of the isogenies are involved. This gives a different approach to the isogeny problem without computing an isogeny.}
\begin{equation}
(\mathcal{D}(a, b, \varphi)\cdot \mathcal{D}(a, b, \varphi)) = 2(ab -\deg(\varphi)),\quad
 (\mathcal{D}(a, b, \varphi)\cdot (x\theta_1 + y\theta_2)) = bx+ay.
\end{equation} 
\end{theorem}
\noindent
Recall that a bilinear form $\beta_d$ is defined by 
$$\beta_d(\varphi,\varphi')= q_{E_1,E_2}(\varphi+ \varphi')- q_{E_1,E_2}(\varphi)-q_{E_1,E_2}(\varphi').$$
Using the  \eqref{intersectionoftwodivisors} above, we can easily calculate the intersection numbers of divisors on abelian product surfaces. Moreover, we can numerically characterize principal polarizations as follows.

\begin{corollary}
\emph{\cite[Corollary 25]{KaniMJ}}\label{polarizationcondition}
Let $\A/K$ be an abelian surface and let
$D = \mathcal{D}(a,b,\varphi) \in \NS(\A)$,
using the notation of \emph{Theorem~\ref{mainisomorphism}}.
Then $D \in \mathcal{P}(\A)$ if and only if $a > 0$ and $ab - \deg(\varphi) = 1$. 
Thus, every principal polarization of $\A$ has the form $\mathcal{D}(n_1,n_2,\varphi)$ 
with $\varphi \in \Hom(E_1,E_2)$ and $n_1,n_2 >0$ with $n_1 n_2- (\deg(\varphi))=1$.
\end{corollary}
\noindent
The relation between the intersection form \eqref{intersectionform} and the degree map \eqref{defn:degreemap} is given by
\begin{align}
    q_{\A}(x, y,\varphi)= xy- q_{E_1,E_2}(\varphi).
\end{align}
where $xy$ denotes the quadratic form defined
by the hyperbolic plane.
Note that $q_{\A}$ is an indefinite integral quadratic form in $\rho=r+2$ variables where $\rho$ is the Picard number of $\A$,  and $r$ is the rank of $\Hom(E_1,E_2)$.

\begin{lemma}\emph{\cite[Lemma 28]{KaniESCI}} 
The determinant of the N\'eron-Severi group of $E_1\times E_2$ with respect to the intersection form is given by
$$ \det(\NS(E_1 \times E_2)) = (-1)^{\rho -1} \det(\Hom(E_1,E_2), \beta_d),$$
where $\rho =\rank(\NS(E_1\times E_2)) = \rank(\Hom(E_1, E_2))+2$.
\end{lemma}

\subsubsection{Polarized N\'eron-Severi group.} 

Let $\A/K$ be an abelian surface and let $\theta \in \mathcal{P}(\A)$ be a principal polarization of $\A$. We define the \emph{polarized N\'eron-Severi group} of $(\A,\theta)$ to be    \[\NS(\A,\theta) \coloneqq \NS(\A)/\Z\theta.\]

Kani discusses in \cite[\S 3]{Kani94} that there is a well-defined map on $\NS(\A,\theta)$, and this defines a positive-definite quadratic form on $\NS(\A,\theta)$. 
Suppose that $\A$ has a principal polarization $\theta\in\mathcal{P}(\A)$. Then the quadratic form $\tilde{q}_{(\A,\theta)}$ on $\NS(\A)$ is
\begin{equation} \label{eq:tildeform}
\tilde{q}_{(\A,\theta)}(D) = (D \cdot \theta)^2 - 2(D \cdot D), \textrm{ for } D \in \NS(\A).
\end{equation}
It is clear to see that $\tilde{q}_{(\A,\theta)}(D + n\theta) = \tilde{q}_{(\A,\theta)}(D)$ for all $n\in \Z$. As a consequence, $\tilde{q}_{(\A,\theta)}$ actually leads to a quadratic form $q_{(\A,\theta)}$ on the quotient module $\NS(\A,\theta)$. 
\begin{definition}\label{RHIdefn}
    Let $(\A,\theta)$ be a principally polarized abelian surface. A
    \emph{refined Humbert invariant}\footnote{A refined Humbert invariant is considered up to equivalences, so mostly we refer it as the refined Humbert invariant up to isometries for a fixed principally polarized abelian surface $(\A,\theta)$} $q_{(\A,\theta)}$ of $(\A,\theta)$ is a positive-definite quadratic form on $\NS(\A,\theta)$, or more precisely, the quadratic module $(\NS(\A, \theta), q_{(\A,\theta)})$, satisfying,
    for all $\cl(D)\in \NS(\A,\theta)$
    \[q_{(\A,\theta)}(D) = (D \cdot\theta)^2 - 2(D\cdot D).\]
\end{definition}

\begin{proposition}\emph{\cite[Lemma 30]{KaniESCI}}\label{prop:detNS}
Let $\rho = \rank(\NS(\A))$. Then the determinant of the quadratic module $(\NS(\A, \theta ), q_{(\A,\theta )})$ is related to that of the N\'eron–Severi group by the formula
$$\det(\NS(\A,\theta), q_{(\A,\theta )}) = \frac{1}{2} (-4)^{\rho -1}\det(\NS(\A), q_{\A}).$$
\end{proposition}

\subsubsection{Irreducibility criterion.}
One of the useful properties of a refined Humbert invariant $q_{(\A,\theta)}$ of $(\A,\theta)$ is the following irreducibility criterion.
\begin{definition}\cite[Satz 2]{Weilsatz}\label{defn:reducpolar}
A polarization $\theta \in \mathcal{P}(\A)$ is called \emph{reducible} (or \emph{decomposable}) if $\theta  =
\cl(E_1 + E_2)$ for some elliptic curves $E_1$ and $E_2$ on $\A$. The set of all reducible polarizations is denoted by $ \mathcal{P}(\A)^{\red}$.
\end{definition}
\begin{proposition}\emph{\cite[Proposition 6]{KaniMJ}} \label{irreducibilitycriterion}
Let $q_{(\A,\theta)}$ the refined Humbert invariant  of the principally polarized abelian surface $(\A,\theta)$, then we have that
\begin{equation} 
 \theta \text{  is reducible if and only if } q_{(\A,\theta)} \text{ represents } 1.
\end{equation}
\end{proposition}
\noindent
Proposition~\ref{irreducibilitycriterion} above allows us to decide whether a principally polarized abelian surface $\A$ is Jacobian of a curve  $(\mathcal{J(C)},\theta_\mathcal{C})$, or a product of two elliptic curves $(E_1\times E_2,\theta_{E_1\times E_2})$.
\subsubsection{A necessary condition.} For a generic principally polarized abelian surface $(\A, \theta)$, there exists a necessary condition for an integral quadratic form appearing as a refined Humbert invariant as follows. 
\begin{theorem}\emph{\cite[Theorem 3.1.1]{harunkirthesis}}

If an integral quadratic form $f$ is equivalent to a refined Humbert invariant $q_{(\A,\theta)}$ for some principally polarized abelian surface $(\A, \theta)/K$, then $f \equiv 0, 1 \pmod 4$.
\end{theorem}

\subsubsection{An equivalent definition.} Later, Kani generalized the definition of a refined Humbert invariant in \cite[Remark 16]{KaniESCI} as follows.
\begin{definition}\label{productpolar}    
Let $\A$ be an abelian surface over a field $K$. If $\A$ has a principal polarization $\theta :\A \rightarrow \hat{\A}$ defined over $K$,
then we define the additive subgroup $\End_{\theta}(\A)$ of the ring $\End(\A) = \End_{K}(\A)$ of $K$-endomorphisms of $\A$ by:
$$\End_{\theta}(\A) = \{\mu \in \End(\A) :\hat{\mu} \circ \theta = \theta \circ \mu \} = \{\mu \in \End(\A) :\mu= \mu' \},$$
where $\mu'= r_{\theta}(\mu) \coloneqq \theta^{-1}\circ \hat{\mu} \circ \theta$. Thus, $\End_{\theta}(\A)$ consists of those endomorphisms
which are symmetric with respect to the Rosati involution $r_{\theta}$ defined by $\theta$.
\end{definition}
\begin{proposition} \emph{\cite [Proposition 14, Remark 16]{KaniESCI}}\label{RHInewdefn}
Let $(\A,\theta)$ be a principally polarized abelian surface over a field $K$ and let $q_{(\A,\theta)}$ be a refined Humbert invariant of $(\A,\theta)$.
Then for every $\mu \in \End_{\theta}(\A)$, we have that 
\[q_{(\A,\theta)}(\mu) = \tr(\mu^2) - \frac{1}{4}(\tr(\mu))^2\]
where  $\tr$ is the usual rational trace of an endomorphism as defined in \emph{\cite[p.182]{mumford1970Abelian}}. 
\end{proposition}
 There exists a close relation between the refined Humbert invariant  $q_{(\A,\theta)}$  $(\A,\theta)$ and the degree map $q_{E_1, E_2}$ when $\A = E_1\times E_2$ is the product surface, as follows.
\begin{lemma} \emph{\cite [Lemma 21]{KaniESCII}} \label{fromRHItodegreemap}
Let $E_1/K$ and $E_2/K$ be elliptic curves over an arbitrary field $K$, and let $\A = E_1\times E_2$ be the product surface with the product polarization $\theta_{\A} =\theta_{E_1} \otimes \theta_{E_2}$. For $a,b\in \Z$, and $\varphi \in \Hom(E_1,E_2)$, then
\begin{equation}\label{RHI=x2+4}
q_{(\A,\theta_{\A})}(\mathcal{D}(a, b, \varphi)) = (a-b)^2+4q_{E_1,E_2}(\varphi),
\end{equation}
$q_{E_1,E_2}$ denotes the degree map on $\Hom(E_1, E_2)$.
\end{lemma}
More general statements can be given related to refined Humbert invariants by using elliptic subcovers and isogeny defects as follows.

\begin{definition}
Let $\mathcal{C}/K$ be a curve of genus $2$. The presentation of $\mathcal{C}/K$ of degree $N$ is the triple $(E, E', \psi)$  which arises from a given elliptic subcover $F:\mathcal{C} \rightarrow E$ of degree $N$. This triple consists of $E$, another (isogenous) elliptic curve $E'/K$, and an
isomorphism $\psi: E[N] \rightarrow E'[N] $ which is an anti-isometry with respect to the  Weil pairing $e_N$. 
\end{definition}

\begin{definition}\label{defect}
Attached to $\psi$, an invariant, \emph{the isogeny defect} $m_{\psi}$ is defined as
\[m_{\psi} \coloneqq \min\{m \geq 1 :[m]\circ \psi = \varphi |_{E[N]} \text{ for some } \varphi \in \Hom(E, E')\}.\]
\end{definition}

\begin{theorem}\emph{\cite[Theorem 3]{KaniESCI}}\label{thm:detRHI}
If $\mathcal{C}/K$ has a presentation $(E, E',\psi)$ of degree $N$ with $\charec(K) \nmid N$ an isogeny
defect $m = m_\psi$, and if $r = \rank(\Hom(E, E'
))\geq 1$, then the refined Humbert invariant $q_{\mathcal{C}}$ is
a positive definite quadratic form of rank $n=r+1$, which satisfies properties
\begin{itemize}
    \item [(i)]$\det(q_{\mathcal{C}}) = 2^{2r+1}m^2\det(q_{E,E'})$.
    \item [(ii)] $q_{\mathcal{C}}$ primitively represents $N^2$.
    \item [(iii)] $q_{\mathcal{C}}(x_1,\dots,x_{r+1}) \equiv 0, 1 \pmod  4$, for all  $x_1,\dots,x_{r+1}\in \Z$.
    \item [(iv)] $q_{\mathcal{C}}(x_1,\dots,x_{r+1}) \neq 1$ for any  $x_1,\dots,x_{r+1}\in \Z$.
\end{itemize}
\end{theorem}

\begin{theorem}\emph{\cite[Theorem 4]{KaniESCI}}\label{defectis1}
If $(E, E',\psi)$ is a presentation of the degree $N$ of a curve $\mathcal{C}/K$ of genus $2$ with $\charec(K)\nmid N$, then $m_\psi = 1$ if and only if $\mathcal{J(C)} \cong E\times E'$.
\end{theorem} 
\noindent
The set $\mathcal{P}(\A, q)$ is defined as all the principal
polarizations $\theta \in \mathcal{P}(\A)$ which are equivalent to the refined Humbert invariant $q_{(\A,\theta)}$
$$\mathcal{P}(\A, q) = \{\theta \in \mathcal{P}(\A) :q_{(\A,\theta)} \sim q\}.$$
The set of reducible polarizations, as in Definition~\ref{defn:reducpolar}, can be written as the union of several sets of the form $\mathcal{P}(\A, q_i)$ by varying $q_i$'s in the following way.
\begin{proposition}\emph{\cite[Proposition 6]{KaniPrincipalPolarizations}} \label{prop:generasum}
If $\A = E \times E'$ is an abelian product surface, then
$$ \mathcal{P}(\A)^{\red} = \coprod_{q\in \Gen(q_{E,E'})}
\mathcal{P}(\A, x^2 \perp 4q)$$ where $\Gen(q)$ is the set of isomorphism classes of integral
quadratic forms $q$  which are genus-equivalent to the integral quadratic form $q_{E,E'}$.
\end{proposition}

\subsection{Superspecial abelian varieties}
An elliptic curve is an abelian variety of dimension $1$. An isogeny between elliptic curves is a surjective homomorphism with finite kernel. An isogeny from an elliptic curve to itself is called an endomorphism, and the set of endomorphisms of $E$ forms the ring $\End(E)$. If two elliptic curves $E_1$,$E_2$ are isogenous $E_1 \sim E_2$, then $\End^{0}(E_1)\simeq \End^{0}(E_2)$.
An elliptic curve $E$ defined over $\overline{\F}_{p}$ is said to be \emph{supersingular} if the endomorphism algebra, $\End^0_{\overline{\F}_p}(E)$, is isomorphic to a definite quaternion algebra  $B_p$ over $\Q$ ramified at $p$ and $\infty$. Here, $B_p = \Q + \Q\mi + \Q\mj + \Q\mi\mj$ for $\mi^2=a$, $\mj^2=b$, and $\mi\mj = -\mj\mi$. In particular, for $p\equiv 11\pmod {12}$ we can consider $(a,b)=(-1,-p)$~\cite[Proposition 1]{EHLMP}.

An abelian variety over $\overline{\F}_{p}$  is said to be \emph{supersingular} if it is isogenous to a product of supersingular elliptic curves over $\overline{\F}_p$ \cite{Oort1974}, and it is said to be \emph{superspecial} if it is isomorphic to a product of supersingular elliptic curves over $\overline{\F}_{p}$  (as an unpolarized abelian variety). A curve $\mathcal{C}$ 
is called \emph{supersingular} (respectively, \emph{superspecial}) if its Jacobian $\A = \mathcal{J(C)}$ is supersingular (respectively, \emph{superspecial}). 

\begin{theorem}\label{unpolsuperspecial}(Deligne, Ogus, Shioda) If $\A/\overline{\F}_p$ is a superspecial abelian
variety with $\dim \A = g > 1$, then $\A\simeq  E^g$ for any supersingular elliptic curve $E$.
\end{theorem} 
For proof, see \cite[Theorem 6.2]{Ogus},\cite[Theorem 3.5]{Shioda}, and \cite[Section 1.6]{OortLiModularSpacesofSuperSingular}. 

In the case of dimension $g=1$, there are many superspecial abelian varieties (i.e., supersingular elliptic curves), each with one principal polarization, but in the case of  $g>1$, there is one superspecial abelian variety with many principal polarizations.

The map in Theorem~\ref{mainisomorphism} (or using Corollary 2.9 of \cite{ibukiyama1986supersingular}) induces a bijection between principal polarizations  $\theta\in \mathcal{P}(E_1\times E_2)$ and positive definite quaternion hermitian matrices with determinant $1$
\begin{align}\label{eq:sspolar}
    \mathcal{P}(E_1\times E_2) =  \left\{\begin{pmatrix}
        u & \alpha \\ \overline{\alpha} & v
    \end{pmatrix} :u,v\in \Z_{>0}, \ \alpha\in \curlyO,\  uv-\alpha\overline{\alpha}=1\right\}
\end{align}
where $\curlyO$ is a maximal order of $B_p$.

The natural question is how to choose a representative in the conjugacy (isomorphism) class of a principal polarization on the superspecial abelian surface, as there are too many of them to use in the computations. This issue has been considered by Hashimoto and Ibukiyama~\cite{HashIbu}. 
\noindent 
Let $\curlyO$ be a maximal order of $B_p$, $\mL(\curlyO)$ be the set of all maximal $\curlyO$-lattices, and $\mL(\curlyO; 0)$  be the principal genus. Then any $\curlyO$-lattice in $B_p^2$ can be written as $\Lambda=(\curlyO, \curlyO)G$ for $G\in \GL_2(B_p)$. 
\begin{proposition}\emph{\cite[Proposition 22]{HashIbu}}
    $\Lambda=(\curlyO, \curlyO)G$ belongs to $\mL(\curlyO;0)$ if and only if $G$ satisfies the condition
    \begin{equation}\label{eq:polarizationcondition}
    GG^{*}=r\begin{pmatrix}
        u & \alpha \\ \bar{\alpha} & v
    \end{pmatrix}; \quad u,v \in \Z_{>0}, \quad \alpha\in \curlyO, \quad uv-\nrd(\alpha)=1, \quad r\in \Q^{\times}_{+}.
\end{equation} 
\end{proposition}
Any maximal $\curlyO$-lattice in $\mL(\curlyO,0)$ is equivalent to a maximal $\curlyO$-lattice. Therefore, one can reduce the problem of finding all representatives of the classes in the $\mL(\curlyO;0)$, to the problem of finding all $(u,v,\alpha)$ satisfying \eqref{eq:polarizationcondition} up to the equivalence by $\GL_2(\curlyO)$.

\begin{lemma}\emph{\cite[Lemma 13]{HashIbu}}\label{conjugacyofpolarizations} With the same notation as above, we have the following information regarding the conjugacy classes of any lattice $\Lambda\in\mL(\curlyO,0)$.
    \begin{itemize}
    \item[(i)] The equivalence class of $\begin{pmatrix}
        u & \alpha \\ \overline{\alpha} & v
    \end{pmatrix}$ only depends on $\alpha \pmod v$ for fixed $v$.
    \item[(ii)]\label{secondpoint}   If $\beta,\beta'\in \curlyO^{\times}$, then $\begin{pmatrix}
        u & \alpha \\ \overline{\alpha} & v
    \end{pmatrix}$ and $\begin{pmatrix}
        u & \beta \alpha\beta' \\ \overline{\beta \alpha \beta'} & v
    \end{pmatrix}$ are equivalent.
    \item[(iii)] $\begin{pmatrix}
        u & \alpha \\ \bar{\alpha} & v
    \end{pmatrix}$ and $\begin{pmatrix}
        u & \bar{\alpha} \\ \alpha & v
    \end{pmatrix}$ are equivalent.
  \end{itemize}
\end{lemma}
Later, in Algorithm~\ref{alg:polz}, we will apply Lemma~\ref{conjugacyofpolarizations} to find representatives in the conjugacy classes of principal polarizations.
By using the basis of $\curlyO$ over $\Z$, we can give an algorithm to find all triples $(u, v, \alpha)$ satisfying the condition in \eqref{eq:polarizationcondition} as follows:
\begin{enumerate}
    \item Let $v=1,2,3, \dots$
    \item For each $v$, find all $\alpha\in \curlyO/(v\curlyO)$ such that $\nrd(\alpha)+1=0 \pmod v$
    \item Compute $u=(\nrd(\alpha)+1)/v$.
\end{enumerate}

\section{Computation}\label{designing}
A refined Humbert invariant $q_{(\A, \theta)}$, Definition~\ref{RHIdefn}, corresponding to a principally polarized superspecial abelian surface $(\A, \theta) $ is an integral quadratic form in $5$ variables, called the \emph{quintic refined Humbert invariant}\footnote{Kani refers to them as quintic quadratic forms, although many people in the literature use quinary integral form. We prefer to use quintic integral forms following Kani's terminology.}. 

Unless stated otherwise, let $p$ be an odd prime, let $E$ be a fixed supersingular elliptic curve defined over $\overline{\F}_{p}$, and let $\A=E\times E$  be a supersingular product surface with principal polarization $\theta$. Assume that $p \equiv 11 \pmod{12}$ and $\End^0_{\overline{\F}_p}(E) = B_p := (-1, -p | \mathbb{Q})$ where  $B_p = \Q + \Q\mi + \Q\mj+ \Q\mi\mj$ where $\mi^2=-1$ and $\mj^2=-p$.

\begin{remark}
The isogeny defect  $m_{\psi} = 1$ if and only if $\psi$ is induced by an isogeny, see \cite[Remark 21]{KaniESCI}. Since we are interested in the product surface $\A=E_1\times E_2$ where $E_1$ and $E_2$ are supersingular elliptic curves, and the fact that supersingular elliptic curves are all isogenous to each other, then the isogeny defect is $m_{\psi}=1$ (also see Theorem~\ref{defectis1}).
\end{remark}

\subsection{Quintic integral quadratic forms as refined Humbert invariants} \label{PPequivalence}
First, we will walk through the two methods available for computing the quintic refined Humbert invariant $q_{(\A,\theta)}$ for a principally polarized superspecial abelian surface $(\A,\theta)$ where $\A=E_1\times E_2$ for supersingular elliptic curves $E_1$ and $E_2$.

Recall from \ref{unpolsuperspecial} that all superspecial abelian surfaces over a field of characteristic $p$ are isomorphic if we ignore polarizations.
Therefore, any principally polarized superspecial abelian surface $(\A,\lambda_\A)$ is isomorphic to some fixed superspecial surface $\A_0= E\times E$ equipped with a suitable principal polarization $\lambda_{\A_0}$.
Explicitly, if $\vartheta  : \A_0 \to \A$ is an unpolarized isomorphism, then we
can take $\lambda_0 = \hat{\vartheta }\,\lambda_\A\,\vartheta $. A construction of Ibukiyama–Katsura–Oort \cite{ibukiyama1986supersingular} encodes a principal polarization $\lambda_0$ on $\A_0$ by a matrix with coefficients in  $\mathcal{O}_0$.  Consider the map
\[
\begin{aligned}
\Upsilon : \mathcal{P}(E\times E) &\longrightarrow \End(\A_0) \\
\lambda_\A &\longmapsto \lambda_0^{-1}\lambda_\A,
\end{aligned}
\]
note that $\lambda_0^{-1}\lambda_\A$ 
may be viewed as an element of $M_2(\mathcal O_0)$. By specializing \cite[Corollary 2.9]{ibukiyama1986supersingular} to principal polarizations,
the map $\Upsilon$ is injective. After translating it with the Deuring correspondence to the quaternion setting, we can define the set of polarizations, and its image is exactly the set
\[
\operatorname{Pol}(\A_0) :=
\left\{
\begin{pmatrix}
        u & \alpha \\ \overline{\alpha} & v
    \end{pmatrix} 
\ \middle|\ 
u,v \in \mathbb Z_{>0},\ r \in \mathcal{O}_0,\ uv - \alpha\overline{\alpha} = 1
\right\}
\ \subset\ \operatorname{GL}_2(\mathcal{O}_0).
\]
Equivalently, $\Upsilon$ gives a bijection between $\mathcal{P}(E\times E)$  and
$\operatorname{Pol}(\A_0)$.

Now, we fix our convention throughout the rest of the computations\footnote{In this computation, any maximal order of $B_p$ can be chosen, and the computation can be modified with respect to the maximal order.}. For $p\equiv 11\pmod {12}$ and $B_p = (-1,-p |\Q)$, we can fix $\curlyO = \Z\langle 1, \beta_1, \beta_2, \beta_3\rangle = 
    \Z\left\langle 1, \mi,\dfrac{\mi + \mj}{2},\dfrac{1+\mi\mj}{2} \right\rangle \cong \End(E)$, where $j(E)=1728$.
Then choosing $\alpha_0 = w_0+x_0\beta_1 +y_0\beta_2 +z_0\beta_3\in \curlyO$, we get the integral quadratic form 
\begin{align}\label{eq:normform}
   g(w_0,x_0,y_0,z_0) &= \nrd(\alpha_0) = w_0^2 + w_0z_0 + x_0^2 + x_0y_0 + \frac{p+1}{4}y_0^2 + \frac{p+1}{4}z_0^2,
\end{align} such that $\det(g) = \det\begin{psmallmatrix} 2 & 0 & 0 & 1 \\ 0 & 2 & 1 & 0 \\ 0 & 1 & \frac{p+1}{2} & 0 \\ 1 & 0 & 0 & \frac{p+1}{2}\end{psmallmatrix} =  p^2$.

We walk through the method for computing the quintic refined Humbert invariant, summarized in Algorithm~\ref{alg:RHI}.
\begin{enumerate}
    \item Fix the principal polarization $\theta$ on the abelian surface $E\times E$ represented by a positive definite hermitian matrix of determinant 1
    \[\theta \coloneqq \begin{pmatrix}
    u_0 & w_0+x_0\beta_1+y_0\beta_2+z_0\beta_3 \\
    w_0+x_0\bar{\beta}_1 + y_0 \bar{\beta}_2 + z_0 \bar{\beta}_3 & v_0 
     \end{pmatrix} = 
     \begin{pmatrix}
    u_0 & \alpha_0 \\
    \overline{\alpha}_0 & v_0 
     \end{pmatrix}\]
     for fixed values of $u_0, v_0$ and $\alpha_0 \coloneqq w_0+x_0\beta_1+y_0\beta_2+z_0\beta_3\in \curlyO$ such that $u_0>0$, $v_0 > 0$, and $u_0v_0 - \nrd(\alpha_0)= 1$, i.e. $u_0v_0 - w_0^2 - w_0z_0 - x_0^2 - x_0y_0 - \frac{p+1}{4}y_0^2 - \frac{p+1}{4}z_0^2 = 1$ (i.e. self-intersection number of $\theta$ is 2). 
\item Choose an arbitrary divisor $D$ on $E\times E$ 
    \[D \coloneqq \begin{pmatrix}
    u & w+x\beta_1 + y\beta_2 + z \beta_3 \\
    w+x\bar{\beta}_1 + y \bar{\beta}_2 + z \bar{\beta}_3 & v 
    \end{pmatrix} \in  \NS(E\times E)\subseteq M_2(\curlyO)\]  
    for $u, v, w, x, y, z \in \mathbb{Z}$. Note that
    \[w+x\beta_1 + y\beta_2 + z \beta_3  = \left(w +\frac{z}{2}\right)+ \left(x+\frac{y}{2}\right)\mi + \frac{y}{2}\mj + \frac{z}{2}\mi\mj.\]
\item Compute the intersection of the divisors $\theta$ and $D$
\begin{align*}
   (D\cdot \theta) &= v_0u+u_0v+2\left(p\frac{z_0z}{4}-\left(w_0+\frac{z_0}{2}\right)\left(w+\frac{z}{2}\right)- \left(x_0+\frac{y_0}{2}\right)\left(x+\frac{y}{2}\right)-p\frac{y_0y}{4}\right)\\
   &= v_0u + u_0v - (2w_0+z_0)w - (2x_0+y_0)x - \left(x_0 +\frac{y_0(p+1)}{2}\right)y - \left(w_0 - \frac{z_0(p-1)}{2}\right)z,\\
    (D\cdot \theta)^2 &= 
    \resizebox{\textwidth}{!}{%
    $\displaystyle
        \begin{aligned}[t]
    & v_0^2 u^2 + 2 u_0 v_0 u v - (4 v_0 w_0 + 2 v_0 z_0)uw - (4 v_0 x_0 + 2 v_0 y_0) u x - (2 v_0 x_0 + v_0 y_0 p + v_0 y_0)uy + (v_0 z_0 p -2 v_0 w_0 - v_0 z_0) u z \\
    & + u_0^2 v^2 - (4 u_0 w_0 + 2 u_0 z_0) v w - (4 u_0 x_0 + 2 u_0 y_0) v x - (2 u_0 x_0 + u_0 y_0 p + u_0 y_0) v y + (u_0 z_0 p - 2 u_0 w_0 - u_0 z_0) v z \\
    & + (4 w_0^2 + 4 w_0 z_0 + z_0^2) w^2 + (8 w_0 x_0 + 4 w_0 y_0 + 4 x_0 z_0 + 2 y_0 z_0) w x + (4 w_0 x_0 + 2 w_0 y_0 p + 2 w_0 y_0 + 2 x_0 z_0 + y_0 z_0 p + y_0 z_0) w y + (4 w_0^2 - 2 w_0 z_0 p + 4 w_0 z_0 - z_0^2 p + z_0^2) w z \\
    & + (4 x_0^2 + 4 x_0 y_0 + y_0^2) x^2 + (4 x_0^2 + 2 x_0 y_0 p + 4 x_0 y_0 + y_0^2 p + y_0^2) x y + (4 w_0 x_0 + 2 w_0 y_0 - 2 x_0 z_0 p + 2 x_0 z_0 - y_0 z_0 p + y_0 z_0) x z \\
    & + (x_0^2 + x_0 y_0 p + x_0 y_0 +  \frac{1}{4} y_0^2 p^2 +  \frac{1}{2} y_0^2 p +  \frac{1}{4} y_0^2) y^2 + (2 w_0 x_0 + w_0 y_0 p + w_0 y_0 - x_0 z_0 p + x_0 z_0 -  \frac{1}{2} y_0 z_0 p^2 +  \frac{1}{2} y_0 z_0) y z 
    \end{aligned}
    $} \\
    &\quad + (w_0^2 - w_0 z_0 p + w_0 z_0 +  \frac{1}{4} z_0^2 p^2 -  \frac{1}{2} z_0^2 p +  \frac{1}{4} z_0^2) z^2.
\end{align*}
\item Compute the self-intersection of the divisor $D$
        \[(D\cdot D) = 2\left(uv - wz - xy - w^2 -x^2\right) - \frac{(p+1)}{2}(z^2 + y^2).\]
\item Compute $\tilde{q}_{(E \times E, \theta)}(D)=(D\cdot \theta)^2 - 2(D\cdot D)$
    \begin{align*}
        \tilde{q}_{(E\times E, \theta)}(D) & = 
        \resizebox{\textwidth}{!}{%
    $\displaystyle
        \begin{aligned}[t]
        & v_0^2 u^2 + (2 u_0 v_0 - 4) u v + (-4 v_0 w_0 - 2 v_0 z_0) u w + (-4 v_0 x_0 - 2 v_0 y_0) u x + (-2 v_0 x_0 - v_0 y_0 p - v_0 y_0) u y + (-2 v_0 w_0 + v_0 z_0 p - v_0 z_0) u z \\
        &  + u_0^2 v^2 + (-4 u_0 w_0 - 2 u_0 z_0) v w + (-4 u_0 x_0 - 2 u_0 y_0) v x + (-2 u_0 x_0 - u_0 y_0 p - u_0 y_0) v y + (-2 u_0 w_0 + u_0 z_0 p - u_0 z_0) v z \\
        & + (4 w_0^2 + 4 w_0 z_0 + z_0^2 + 4) w^2 + (8 w_0 x_0 + 4 w_0 y_0 + 4 x_0 z_0 + 2 y_0 z_0) w x + (4 w_0 x_0 + 2 w_0 y_0 p + 2 w_0 y_0 + 2 x_0 z_0 + y_0 z_0 p + y_0 z_0) w y + (4 w_0^2 - 2 w_0 z_0 p + 4 w_0 z_0 - z_0^2 p + z_0^2 + 4) w z \\
        & + (4 x_0^2 + 4 x_0 y_0 + y_0^2 + 4) x^2 + (4 x_0^2 + 2 x_0 y_0 p + 4 x_0 y_0 + y_0^2 p + y_0^2 + 4) x y + (4 w_0 x_0 + 2 w_0 y_0 - 2 x_0 z_0 p + 2 x_0 z_0 - y_0 z_0 p + y_0 z_0) x z \\
        & + (x_0^2 + x_0 y_0 p + x_0 y_0 + \frac{1}{4} y_0^2 p^2 + \frac{1}{2} y_0^2 p + \frac{1}{4} y_0^2 + p + 1) y^2 + (2 w_0 x_0 + w_0 y_0 p + w_0 y_0 - x_0 z_0 p + x_0 z_0 - \frac{1}{2} y_0 z_0 p^2 + \frac{1}{2} y_0 z_0) y z     \end{aligned}
    $} \\
        &\quad+ (w_0^2 - w_0 z_0 p + w_0 z_0 + \frac{1}{4} z_0^2 p^2 - \frac{1}{2} z_0^2 p + \frac{1}{4} z_0^2 + p + 1) z^2\\
        & \coloneqq \frac{1}{2} X^t A X 
         \end{align*}
         where $X = \begin{psmallmatrix} u \\ v \\ w \\ x \\ y \\ z \end{psmallmatrix}$ and $A$ is the coefficient matrix given by
          \begin{align}\label{the-matrix}
          A =\resizebox{\textwidth}{!}{%
$
\setcounter{MaxMatrixCols}{20} 
\setlength{\arraycolsep}{1pt}   
\begin{pmatrix}
2v_0^2 & 2u_0v_0 - 4 & -4v_0w_0 - 2v_0z_0 & -4v_0x_0 - 2v_0y_0 & -2v_0x_0 - v_0y_0p - v_0y_0 & -2v_0w_0 + v_0z_0p - v_0z_0 \\
2u_0v_0 - 4 & 2u_0^2 & -4u_0w_0 - 2u_0z_0 & -4u_0x_0 - 2u_0y_0 & -2u_0x_0 - u_0y_0p - u_0y_0 & -2u_0w_0 + u_0z_0p - u_0z_0 \\
-4v_0w_0 - 2v_0z_0 & -4u_0w_0 - 2u_0z_0 & 8w_0^2 + 8w_0z_0 + 2z_0^2 + 8 & 8w_0x_0 + 4w_0y_0 + 4x_0z_0 + 2y_0z_0 & 4w_0x_0 + 2w_0y_0p + 2w_0y_0 + 2x_0z_0 + y_0z_0p + y_0z_0 & 4w_0^2 - 2w_0z_0p + 4w_0z_0 - z_0^2p + z_0^2 + 4 \\
-4v_0x_0 - 2v_0y_0 & -4u_0x_0 - 2u_0y_0 & 8w_0x_0 + 4w_0y_0 + 4x_0z_0 + 2y_0z_0 & 8x_0^2 + 8x_0y_0 + 2y_0^2 + 8 & 4x_0^2 + 2x_0y_0p + 4x_0y_0 + y_0^2p + y_0^2 + 4 & 4w_0x_0 + 2w_0y_0 - 2x_0z_0p + 2x_0z_0 - y_0z_0p + y_0z_0 \\
-2v_0x_0 - v_0y_0p - v_0y_0 & -2u_0x_0 - u_0y_0p - u_0y_0 & 4w_0x_0 + 2w_0y_0p + 2w_0y_0 + 2x_0z_0 + y_0z_0p + y_0z_0 & 4x_0^2 + 2x_0y_0p + 4x_0y_0 + y_0^2p + y_0^2 + 4 & 2x_0^2 + 2x_0y_0p + 2x_0y_0 + \frac{1}{2}y_0^2p^2 + y_0^2p + \frac{1}{2}y_0^2 + 2p + 2 & 2w_0x_0 + w_0y_0p + w_0y_0 - x_0z_0p + x_0z_0 - \frac{1}{2}y_0z_0p^2 + \frac{1}{2}y_0z_0 \\
-2v_0w_0 + v_0z_0p - v_0z_0 & -2u_0w_0 + u_0z_0p - u_0z_0 & 4w_0^2 - 2w_0z_0p + 4w_0z_0 - z_0^2p + z_0^2 + 4 & 4w_0x_0 + 2w_0y_0 - 2x_0z_0p + 2x_0z_0 - y_0z_0p + y_0z_0 & 2w_0x_0 + w_0y_0p + w_0y_0 - x_0z_0p + x_0z_0 - \frac{1}{2}y_0z_0p^2 + \frac{1}{2}y_0z_0 & 2w_0^2 - 2w_0z_0p + 2w_0z_0 + \frac{1}{2}z_0^2p^2 - z_0^2p + \frac{1}{2}z_0^2 + 2p + 2 \\
\end{pmatrix}
$
}\end{align}
        with $\det(A) = 2^{10} p^2\left(u_0v_0-\nrd(\alpha_0)-1\right) = 0$.
\item Find a $6\times 5$ matrix $T$ whose entries are integers such that the GCD of its $5\times 5$ minors is 1 and $T^t A T$ is the coefficient matrix of the positive-definite quintic form (refined Humbert invariant $q_{(E_1 \times E_2, \theta)})$
    \[q_{(E \times E, \theta)}(D)= \frac{1}{2} \begin{pmatrix} t_0 & t_1 & t_2 & t_3 & t_4  \end{pmatrix} (T^t A T) \begin{pmatrix} t_0 \\ t_1\\t_2 \\ t_3 \\ t_4 \end{pmatrix}.\]
    Moreover, since the Picard number $\rho(\A)=6$, using Propositon~\ref{prop:detNS} we get 
 \begin{align}\label{eq:29p2}
     \det(T^tAT)=\frac{1}{2}(-4)^{\rho-1}\det(q_{E \times E}) = \frac{1}{2}(-4)^{6-1} (-p^2)= 2^9p^2
 \end{align}
    where the intersection form $q_{\A}=q_{E\times E}=\frac{1}{2}(D\cdot D) = uv - wz - xy - w^2 -x^2 - \frac{(p+1)}{4}(z^2 + y^2)$.
\end{enumerate}

We observe that $\tilde{q}_{(E\times E,\theta)}$ above is a positive semi-definite form and applying\footnote{Ours is a very special case:
\url{https://github.com/Nemocas/Nemo.jl/pull/2011}.} Simon's \textsf{indefiniteLLL}\footnote{\url{https://github.com/thofma/Hecke.jl/blob/master/src/QuadForm/indefiniteLLL.jl}} algorithm~\cite{simon,watkins} quickly gives us positive definite quintic form $q_{(E\times E,\theta)}$. Moreover, we can ensure with the irreducibility criterion, Proposition~\ref{irreducibilitycriterion} above, that the quintic form we obtained corresponds to a superspecial abelian surface $\A=E\times E$ over $\overline{\F}_p$ with a reducible principal polarization $\theta\in \mathcal{P}^{\red}(E\times E)$ by using Kannan-Fincke-Pohst \textsf{minVector}\footnote{\url{https://github.com/thofma/Hecke.jl/blob/master/src/QuadForm/Enumeration.jl}} algorithm~\cite{kannan,FinckePohst,XavierDamien} to check that 1 is the minimum value it represents. Therefore, Algorithm~\ref{alg:RHI} below lets us compute the quintic refined Humbert invariants for principally polarized superspecial abelian surfaces.

\begin{algorithm}[!ht]\label{alg:RHI}
  \caption{$\mathsf{RHI}(p,\theta)$}
  \vspace{.2ex}
  \Input{odd prime $p\equiv 11\pmod {12}$ such that $B_p=(-1,-p|\Q)$ and polarization $\theta = \begin{psmallmatrix}u_0 & \alpha_0 \\\overline{\alpha}_0& v_0\end{psmallmatrix}$ where $\alpha_0\coloneqq w_0+x_0\beta_1 + y_0\beta_2+z_0\beta_3\in \Z+\Z\beta_1 + \Z\beta_2 + \Z\beta_3$.}
  \Output{coefficient matrix of quintic refined Humbert invariant $q_{(E\times E,\theta)}$.}
  \vspace{.4ex}
  $A \leftarrow \text{the $6\times 6$ matrix as in \eqref{the-matrix}}$\;
  \tcp{the coefficient matrix of $\tilde{q}_{(A,\theta)}$ with $\det(A)=2^{10}p^2\left(u_0v_0-\nrd(\alpha_0)-1\right)=0$}
  $A' \leftarrow \mathsf{indefiniteLLL}(A)$\;
  \tcp{$A$ is $6\times 6$ positive semi-definite; $A'$ is $5\times 5$ positive definite.}
  \If{$\det(A') = 2^{9} p^2$}{
    \tcp{Proposition~\ref{prop:detNS} for $\rho = 6$ and $\det(q_{A}) = -p^2$}
    $L \leftarrow$ integer lattice with Gram matrix $A'/2$\;
    \If{$\mathsf{minVector}(L) = 1$}{\label{alg:RHI:5}
      \tcp{Proposition~\ref{irreducibilitycriterion} and \cite[\S 2]{FinckePohst}}
      \Return{$A'$}\;\label{alg:RHI:6}
    }
  }
  \Return{$0$}\;
\end{algorithm}

\begin{remark}\label{rem:q4}
    For $\theta = \begin{psmallmatrix} 1 & 0 \\ 0 & 1 \end{psmallmatrix}$ we get $T= \begin{psmallmatrix}
            0 & 0 & 0 & 0 & 0 \\
            1 & 0 & 0 & 0 & 0 \\
            0 & 0 & 1 & 0 & 0 \\
            0 & 1 & 0 & 0 & 0 \\
            0 & 0 & 0 & 1 & 0 \\
            0 & 0 & 0 & 0 & 1
            \end{psmallmatrix}$, such that $\det(T^tAT)=2^{9}p^2$, leading to 
            \[q_5(t_0,t_1,t_2,t_3,t_4) := q_{(E \times E, \theta)}(D)=t_0^2+4\left(t_1^2 + t_1t_3 + t_2^2 + t_2t_4 + \frac{p+1}{4}t_3^2 + \frac{p+1}{4}t_4^2\right).\]
Note that $q_{(E \times E, \theta)}$ satisfies the reducibility condition stated in \eqref{irreducibilitycriterion}.  Therefore, as per Lemma~\ref{fromRHItodegreemap}, the degree map (degree form) is 
\begin{align}\label{eq:q4}
    q_4(t_1,t_2,t_3,t_4) := q_{E,E} = t_1^2 + t_1t_3 + t_2^2 + t_2t_4 + \frac{p+1}{4}t_3^2 + \frac{p+1}{4}t_4^2
\end{align}
that matches the norm form given in \eqref{eq:normform}. By Proposition~\ref{prop:generasum}, $\Gen(q_4)$ contains all unique degree forms and leads to all unique reducible refined Humbert invariants. It is computed via Kneser's neighbor method\footnote{\url{https://github.com/thofma/Hecke.jl/blob/master/src/QuadForm/Quad/GenusRep.jl}}~\cite{Kneser,Kirschmer,VoightKneser}.

\end{remark}

\subsection{Principal polarizations}
We aim to determine all possible quintic refined Humbert invariants that are not equivalent. However, two distinct principal polarizations can lead to the same quintic refined Humbert invariant. Indeed, we verified that for a given prime $p$, there are examples of principally polarized abelian superspecial surfaces $(\A,\theta)$ and $(\A',\theta')$ leading to the equivalent refined Humbert invariants $q_{(\A,\theta)}\sim q_{(\A',\theta')}$ in the experiments by using the count in Proposition~\ref{prop:generasum}.

We fix the definite quaternion algebra
$B_p=(-1,-p|\Q)$ and
$p \equiv 11 \pmod{12}$, and  a maximal order $\mathcal{O}\subset B_p$ with a chosen $\mathbb{Z}$-basis $\{1,\beta_1,\beta_2,\beta_3\}$.
A principal polarization on a superspecial surface corresponds to a positive-definite unimodular Hermitian matrix \eqref{eq:polarizationcondition}, and we enumerate these matrices up to congruence. 
One can get all product polarizations $\begin{psmallmatrix} u_0 & \alpha_0 \\ \overline{\alpha}_0 & v_0 \end{psmallmatrix}$ by varying $u_0,v_0\in \Z_{>0}$ and taking $\alpha_0=w_0+x_0\beta_1 + y_0\beta_2 + z_0\beta_3\in \Z\langle 1, \beta_1, \beta_2,\beta_3\rangle$ such that $0\leq w_0,x_0,y_0,z_0\leq v_0-1$ and $\nrd(w_0+x_0\beta_1 +y_0\beta_2 +z_0\beta_3) = u_0v_0 -1$. We start with eliminating easy isomorphic polarizations by Lemma~\ref{conjugacyofpolarizations}.

We first recall some basics of quaternion hermitian forms from \cite{Shimura}. Viewing $B_p^2$ as a left vector space over $B_p$, the \emph{binary definite quaternion Hermitian form} on $B_p^2$ is defined as the unique form up to isometry, and explicitly given  by
$h(x,y)=\sum_{i=1}^{2}x_i\overline{y_i}
 =x_1\overline{y_1}+x_2\overline{y_2}$,
for vectors $x=(x_1,x_2), y=(y_1,y_2)\in B_p^2$ where $\overline{\phantom{x}}$ is the canonical involution of $B_p$.

The enumeration of principal polarizations on a superspecial abelian surface in Algorithm~\ref{alg:polz} relies on the following criterion for determining whether two principal polarizations are isomorphic. The method is adapted from Chisholm’s PhD thesis \cite[§3.8.1]{SarahChisholm2014}, itself based on \cite[Lemma 6.2]{GreenbergVoight2014} of Greenberg and Voight, specialized to a maximal order $\curlyO\subset B_p$.
Let $h_\theta$ be the quaternion binary Hermitian form associated with a principal polarization $\theta$ on $\curlyO^2\cong \Z^8$, and fix a $\Q$-basis $\{a_1,\ldots,a_4\}$ of $B_p$. Now, we define the associated integral trace forms in 8 variables by \begin{equation}\label{traceforms}
F_i(x,y)=\trd(a_i h_\theta(x,y)) \text{ for } i=1,\ldots,4.
\end{equation} Two principal polarizations $\theta$ and $\theta'$ are isomorphic if and only if their associated $4$-tuples $(F_1,F_2,F_3,F_4)$ and  $(F'_1,F'_2,F'_3,F'_4)$ of auxiliary trace forms are simultaneously isometric, that is, if there exists a $P\in\GL_8(\mathbb Z)$ carrying each $F_i$ to the corresponding trace form of the other polarization. This is the simultaneous-isometry criterion in \cite[Lemma~6.2]{GreenbergVoight2014}, which controls the $\curlyO$-module structure.
Therefore, Algorithm~\ref{alg:polz} below gives us a finite list of principal polarizations that will lead to an input $\theta$ to $\tilde{q}_{(E\times E,\theta)}$ given in \eqref{eq:tildeform}.

\begin{algorithm}[!ht]
  \caption{$\mathsf{polz}(p,V_{\max},U_{\max},T_{\max})$}
  \label{alg:polz}

  \Input{prime $p\equiv 11 \pmod{12}$; bounds
  $V_{\max},U_{\max}\in\mathbb{Z}_{>0}$; theta coefficients cutoff
  $T_{\max}\in\mathbb{Z}_{>0}$.}

  \Output{A list $\mathrm{Reps}$ of representatives of principal
  polarizations on $\mathcal{A}=E\times E$ up to equivalence, where
  $\theta=\begin{psmallmatrix}u&\alpha\\
  \overline{\alpha}&v\end{psmallmatrix}$ with $u,v>0$,
  $\alpha\in\mathcal{O}$, and $uv-\nrd(\alpha)=1$.}

  $\mathrm{Reps}\leftarrow\texttt{[]}$\;
  $\mathrm{Packets}\leftarrow\texttt{[]}$\;
  $\mathrm{Keys}\leftarrow\texttt{[]}$\;

  \For{$v\leftarrow 1$ \KwTo $V_{\max}$}{
    \ForEach{$\alpha'\in\mathcal{O}/v\mathcal{O}$ with
    $\nrd(\alpha')\equiv -1\pmod v$}{
      
      $\alpha\leftarrow\mathsf{ChooseLift}(\alpha',v)$\;\label{alg:polz:6}
      $u\leftarrow\bigl(\nrd(\alpha)+1\bigr)/v$\;

      \If{$0<u\leq U_{\max}$}{
        $\theta\leftarrow
        \begin{psmallmatrix}
          u & \alpha\\
          \overline{\alpha} & v
        \end{psmallmatrix}$\;

        \For{$i\leftarrow 1$ \KwTo $4$}{
          $F_i\leftarrow\mathsf{AuxiliaryTraceGram}(\theta,a_i)$\;
        }

        $(F_1',P)\leftarrow
        \mathsf{LLLGramWithTransform}(F_1)$\;

        \For{$i\leftarrow 2$ \KwTo $4$}{
          $F_i'\leftarrow PF_iP^t$\;
        }

        $\mathcal{F}_\theta'\leftarrow
        (F_1',F_2',F_3',F_4')$\;

        $\mathrm{key}\leftarrow
        \bigl(
        \mathsf{LocalGenus}_{\{2,p\}}(F_1'),
        \mathsf{ThetaSeriesInitials}(F_1',T_{\max}),
        \min(F_1'),
        \kappa(F_1')
        \bigr)$\;

        $\mathrm{isNew}\leftarrow\mathsf{true}$\;

        \For{$j\leftarrow 1$ \KwTo $\lvert\mathrm{Packets}\rvert$}{
          \If{$\mathrm{key}=\mathrm{Keys}[j]$}{
            \If{$\mathsf{ExactIsometric}
            \bigl(F_1',(\mathrm{Packets}[j])_1\bigr)$}{
              \If{$\mathsf{ExactSimultaneousIsometric}\label{alg:polz:20}
              \bigl(\mathcal{F}_\theta',\mathrm{Packets}[j]\bigr)$}{
                $\mathrm{isNew}\leftarrow\mathsf{false}$\;
                \textbf{break}\;\label{alg:polz:23}
              }
            }
          }
        }

        \If{$\mathrm{isNew}$}{
          append $\theta$ to $\mathrm{Reps}$\;
          append $\mathcal{F}_\theta'$ to $\mathrm{Packets}$\;
          append $\mathrm{key}$ to $\mathrm{Keys}$\;
        }
      }
    }
  }

  \Return{$\mathrm{Reps}$}\;
\end{algorithm}

\subsection{Counting distinct forms}\label{sec:classnumbers}
Knowing the expected number of unique (non-isometric) refined Humbert invariants will significantly improve the runtime of Algorithm~\ref{alg:allRHI} since it will allow us to stop computing new refined Humbert invariants once the target number of non-isometric quintic integral quadratic forms is found. Therefore, here we discuss our attempt to get a bound on this number.

When $\A$ is a supersingular elliptic curve, Deuring showed that $\A$ has a model defined over $\overline{\mathbb{F}}_{p}$ and the class number $\mathbf{h}$ of $\End(\A)$ is calculated by Eichler \cite{eichlerclassnumberh}, Deuring \cite{deuringcorrespondence}, and Igusa \cite{igusaclassnumber} as
\begin{align} \label{h:classnumber}
    \mathbf{h} & = \frac{p-1}{12} + \frac{1}{4}\left(1-\legendre{-1}{p}\right) + \frac{1}{3}\left(1-\legendre{-3}{p}\right).
\end{align}

 The calculation of the number of isomorphism classes of principal polarizations on an abelian surface $\A$, especially the number of isomorphism classes of smooth genus $2$ curves lying on $\A$, was calculated by Ibukiyama, Katsura, and Oort \cite{ibukiyama1986supersingular} in 1986 when $\A = E\times E'$, where $E$ and $E'$ are supersingular elliptic curves. 
\begin{theorem}\emph{\cite[Theorem 2.10]{ibukiyama1986supersingular}}
    Let $B_p^2$ be a left $B_p$-vector space. The number of principal polarizations on $\A=E\times E$ up to automorphisms of $\A$ is equal to the class number $\mathbf{H}$ of the principal genus of the quaternion Hermitian space $B_p^2$.
\end{theorem}

Hashimoto and Ibukiyama \cite[p.~1]{HashIbu} gave a formula for $\mathbf{H}$. Later, Katsura and Oort gave \cite[Theorem 3.3]{katsura-oort} simpler version of the formula for $\mathbf{H}$ as follows:
\begin{align}\label{eq:Hshort}
    \begin{split}
    \mathbf{H} &= \frac{(p-1)(p+12)(p+23)}{2880} + \frac{2p+13}{96}\left(1-\legendre{-1}{p}\right)  + \frac{p+11}{36} \left(1-\legendre{-3}{p}\right)\\
    & + \frac{1}{8}\left(1-\legendre{-2}{p}\right) + \frac{1}{12}\left(1-\legendre{-3}{p}\right)\left(1-\legendre{-1}{p}\right) \\
    & +\begin{cases}
    &0 \quad p\equiv 1, 2, 3 \pmod 5\\
    &\frac{4}{5} \quad p\equiv 4 \pmod 5.
    \end{cases}
\end{split}
\end{align}

Neither $\mathbf{H}$ nor $\mathbf{\frac{h(h+1)}{2}}$ is an asymptotically tight bound for the number of unique refined Humbert invariants (up to isometry). The $\mathbf{H}$ serves as an upper bound on the number of all quintic refined Humbert invariants of a superspecial surface, while $\mathbf{\frac{h(h+1)}{2}}$ is an upper bound on the refined Humbert invariants constructed with reducible polarizations; Definition~\ref{defn:reducpolar}. If we take two supersingular product surfaces  $\A=E_1\times E_2$ and $\A'=E'_1 \times E'_2$, then notice that two different principally polarized superspecial product abelian surfaces $(\A,\theta)$ and $(\A',\theta')$ can have the same refined Humbert invariants $q_{(\A,\theta)}\sim q_{(\A',\theta')}$ (up to equivalence).  This means that the number of principally polarized superspecial abelian surfaces over $K$ is not equal to the number of refined Humbert invariants.

Algorithm~\ref{alg:allRHI} combines Algorithms~\ref{alg:polz}
and~\ref{alg:RHI} into a single pipeline.
The first step calls $\mathsf{dynamicPolz}(p)$, which executes
Algorithm~\ref{alg:polz} with a single bound $A_{\max}$ in place of
the separate bounds $V_{\max}$ and $U_{\max}$, with $T_{\max}$ fixed
at a small constant, growing $A_{\max}$ in stages until the stopping
condition of total $|\mathrm{Reps}| = \mathbf{H}$ is met.
We conjecture that $A_{\max} = O(p)$ at termination.
The resulting list $\Theta$ of all $\mathbf{H}$ principal
polarizations is then passed to Algorithm~\ref{alg:RHI} once per
polarization.

Of the $\mathbf{H}$ calls to $\mathsf{RHI}(p,\theta)$, only
$\mathbf{h}(\mathbf{h}+1)/2$ return a non-zero output, where
$\mathbf{h}$ is given by~\eqref{h:classnumber}.
These correspond to the refined Humbert invariants with reducible polarizations.
For each non-zero output $A$, Algorithm~\ref{alg:allRHI} applies a
two-step uniqueness check.
First, it tests exact matrix equality $A = A'$ against every
$A' \in Q$.
If that fails, it constructs the integer lattices $L_A$ and $L_{A'}$
with Gram matrices $A/2$ and $A'/2$ respectively, and calls
$\mathsf{isIsometric}(L_A, L_{A'})$ via the Plesken--Souvignier
algorithm\footnote{\url{https://github.com/thofma/Hecke.jl/blob/master/src/QuadForm/Morphism.jl}}~\cite{isometry}.
The list $Q$ at termination collects all refined Humbert
invariants with reducible polarizations up to isometry, and by Remark~\ref{rem:q4}, $|Q|$ equals
the number of genera of $q_4$ defined in~\eqref{eq:q4}.

\begin{algorithm}[!ht]
  \caption{$\mathsf{allRHI}(p)$}\label{alg:allRHI}
  \Input{odd prime $p\equiv 11\pmod{12}$ such that $B_p=(-1,-p|\Q)$.}
  \Output{coefficient matrices of all unique refined Humbert invariants for polarizations $\theta = \begin{psmallmatrix}u & \alpha_0 \\\overline{\alpha}_0& v\end{psmallmatrix}$.}
  $Q\leftarrow \texttt{[]}$\;
  \tcp{initialize the list of unique refined Humbert invariants}
  $\Theta \leftarrow \mathsf{dynamicPolz}(p)$\;
  \tcp{running Algorithm~\ref{alg:polz} for various choices of parameters until $\mathbf{H}$ polarizations are obtained}
  \ForEach{$\theta\in \Theta$}{
    $A\leftarrow \mathsf{RHI}(p,\theta)$\;
    \tcp{running Algorithm~\ref{alg:RHI} for each polarization}
    \If{$A\neq 0$}{
      $s \leftarrow \mathrm{true}$\;
      \tcp{assume unique}
      \ForEach{$A'\in Q$}{
        \If{$A = A'$}{
          $s\leftarrow \mathrm{false}$\;
          \tcp{not unique}
          \textbf{break}\;
        }
      }
      $L_A\leftarrow$ integer lattice with Gram matrix $A/2$\;
      \If{$s = \mathrm{true}$}{
        \ForEach{$A'\in Q$}{
          $L_{A'}\leftarrow$ integer lattice with Gram matrix $A'/2$\;
          \If{$\mathsf{isIsometric}(L_A,L_{A'})= \mathrm{true}$}{
            $s \leftarrow \mathrm{false}$\;
            \tcp{not unique}
            \textbf{break}\;
          }
        }
      }
      \If{$s = \mathrm{true}$}{
        $Q.\text{append}(A)$\;
      }
    }
  }
  \Return{$Q$}\;
\end{algorithm}

\subsection{Complexity analysis}
We analyze the runtime and space complexity of
Algorithms~\ref{alg:polz}, \ref{alg:RHI}, and~\ref{alg:allRHI} in
the order in which Algorithm~\ref{alg:allRHI} invokes them.
Analyzing Algorithms~\ref{alg:polz} and~\ref{alg:RHI} first
establishes the costs that feed directly into the analysis of
Algorithm~\ref{alg:allRHI}.

The implementation of Algorithm~\ref{alg:polz} departs from the
pseudocode in three ways that affect complexity.
First, the separate bounds $V_{\max}$ and $U_{\max}$ are collapsed into
a single bound $A_{\max}$, since every candidate is normalized to
$u \leq v$, so bounding $v \leq A_{\max}$ automatically bounds $u$.
Second, rather than iterating over residue classes
$\alpha' \in \mathcal{O}/v\mathcal{O}$ with a \textsf{ChooseLift}
step~\ref{alg:polz:6}, the implementation iterates directly over the order elements of
exact reduced norm $N = uv - 1$ per $(u,v)$ pair, selecting one
representative per equivalence class under
$\alpha \sim \beta\alpha\beta'$ for $\beta, \beta' \in
\mathcal{O}^\times$, as given by Lemma~\ref{conjugacyofpolarizations}(ii).
Third, with $T_{\max}=6$ fixed,
all theta-series computations cost $O(T_{\max}^4) = O(1)$ per candidate.
In practice, Algorithm~\ref{alg:polz} is executed as
$\mathsf{dynamicPolz}$.
Under the conjecture $A_{\max} = O(p)$, the largest
reduced norm arising at termination is $N \leq A_{\max}^2 - 1 = O(p^2)$,
and the total number of principal candidates, with results at each
$(u,v)$ pair cached so no work is repeated across stages, is
\[
    C = \sum_{1 \leq u \leq v \leq A_{\max}} O(uv)
      = O(A_{\max}^4) = O(p^4).
\]

The dominant cost is enumerating all $\alpha \in \mathcal{O}$ with
$\mathrm{nrd}(\alpha) = N = uv - 1$ for each $(u,v)$ pair, carried out
via the Fincke--Pohst algorithm~\cite{FinckePohst} applied to a
rank-$4$ integral quadratic form encoding twice the reduced norm on
$\mathcal{O}$.
In the integral basis of $\mathcal{O}$ in $B_p = (-1,-p|\mathbb{Q})$,
two of the four coordinate directions of this form are scaled by
$q_0 = (p+1)/4$, which is an integer since $p \equiv 11 \pmod{12}$
implies $p \equiv 3 \pmod{4}$.
The Fincke--Pohst algorithm exploits this through its Cholesky
preprocessing: those two directions have true extent
$O(\sqrt{N/q_0})$ rather than $O(\sqrt{N})$, reducing the node count
per $(u,v)$ pair from the naive $O(N^2)$ of a brute-force scan to
$O(N^2/q_0) = O(N^2/p)$.
Summing over all distinct norms $N$, each enumerated at most once by
caching,
\[
    T_{\mathrm{norms}}
    = \sum_{N=1}^{O(p^2)} O\!\left(\frac{N^2}{p}\right)
    = O(p^5).
\]
The remaining per-candidate costs are subordinate.
LLL reduction of the auxiliary form $F_1$ uses the Nguyen--Stehl\'e
algorithm~\cite{NS09} on a fixed-rank Gram matrix with
$O(\log p)$-bit entries, costing $O(\log p)$ per candidate and
$O(p^4 \log p)$ in total.
Computing the invariant key, consisting of local genus symbols at
$\{2,p\}$ via Jordan decomposition,
theta-series initials, minimum, and kissing number $\kappa$, costs $O(\log p)$
per candidate and $O(p^4 \log p)$ in total~\cite{conway2013sphere}.
The exact isometry test via the Plesken--Souvignier
algorithm~\cite{isometry} is called only between key-matching
candidates, and since the invariant key separates almost all
non-isometric pairs, the aggregate isometry cost is empirically
$O(p^4)$.
The overall expected runtime of $\mathsf{dynamicPolz}$ under the
conjecture $A_{\max} = O(p)$ stated above is therefore $O(p^5)$,
dominated by Fincke--Pohst enumeration, with space complexity $O(p^4)$ to store
all candidates and their auxiliary forms across stages.

Algorithm~\ref{alg:RHI} takes a fixed prime $p$ and a single
polarization $\theta$ as input, so all costs below are per-call.
The first step constructs the $6\times 6$ positive semi-definite
matrix $A$ over $\mathbb{Q}$, whose entries are polynomial in the
polarization parameters $(u_0,v_0,w_0,x_0,y_0,z_0)$ and $p$.
Under the conjecture $A_{\max} = O(p)$, we have
$u_0, v_0 = O(p)$, and the principal condition
$u_0 v_0 - \mathrm{nrd}(\alpha_0) = 1$ forces $y_0, z_0 = O(\sqrt{p})$
and $w_0, x_0 = O(p)$, so the largest entries of $A$ scale as $O(p^2)$
and carry $O(\log p)$ bits each.
Assembly of $A$ costs $O(\log p)$.

Since $A$ has rank $5$, Simon's
$\mathsf{indefiniteLLL}$~\cite{simon,watkins} reduces $A/2$ to a
form whose sixth row and column vanish, and $A'$ is the top-left
$5\times 5$ positive definite submatrix of the result.
The input is a rank-$6$ Gram matrix with $O(\log p)$-bit entries,
so $\mathsf{indefiniteLLL}$ runs in $O(\log p)$ by the standard LLL
complexity analysis.
The output $A'$ satisfies $\det(A') = 2^9 p^2$, with entries
scaling as $O(p)$.
The integer lattice $L$ with Gram matrix $A'/2$ is then a rank-$5$
positive definite lattice with $\det(L) = O(p^2)$.
$\mathsf{minVector}(L)$ is computed via the Fincke--Pohst
algorithm with Cholesky preprocessing, as in~\cite[\S 2]{FinckePohst}.
By the LLL--Minkowski bound, the minimum of $L$ is at most
$O(p^{4/5})$, and the Fincke--Pohst node count for a rank-$5$
lattice at this bound is
$O\!\left((p^{4/5})^{5/2}/\sqrt{\det(L)}\right) = O(p)$.
All other per-call costs are $O(\log p)$ or smaller.
The per-call runtime of Algorithm~\ref{alg:RHI} is therefore $O(p)$,
dominated by the Fincke--Pohst minimum computation, with space
complexity $O(1)$ as all matrices are of fixed size and no caches
are maintained.

Algorithm~\ref{alg:allRHI} proceeds in two phases as described above.
The first phase runs $\mathsf{dynamicPolz}(p)$ at a cost of $O(p^5)$
as established above.
The second phase iterates over the resulting list $\Theta$ of
$\mathbf{H}$ polarizations.
The leading term of $\mathbf{H}$ is $(p-1)(p+12)(p+23)/2880$, so
$|\Theta| = O(p^3)$.
For each $\theta \in \Theta$, one call to $\mathsf{RHI}(p,\theta)$
costs $O(p)$ as established above, giving a total cost of $O(p^4)$
for all $\mathsf{RHI}$ calls.
Of these, $\mathbf{h}(\mathbf{h}+1)/2 = O(p^2)$ return a non-zero
output, and the remaining $\mathbf{H} - \mathbf{h}(\mathbf{h}+1)/2
= O(p^3)$ calls return $0$ and contribute no further cost.

For each non-zero output $A$, exact matrix equality tests against
every $A' \in Q$ cost $O(1)$ per test since $A$ and $A'$ are
fixed-size $5 \times 5$ integer matrices, giving a total
exact-equality cost of $O(p^2 \cdot |Q|)$ over all non-zero outputs.
If exact equality fails, $\mathsf{isIsometric}$ is called via the
Plesken--Souvignier algorithm~\cite{isometry}, with empirical total
cost $O(p^2 \cdot |Q|)$ over all non-zero outputs.
The overall runtime of Algorithm~\ref{alg:allRHI} is therefore
$O(p^5)$, dominated by $\mathsf{dynamicPolz}$, with all subsequent
costs subordinate under the conjectured bounds.
The space complexity is $O(p^3)$ to store $\Theta$, plus $O(|Q|)$
for the list of unique invariants.

\subsection{Experiment}\label{experiment-data}

We run Algorithm~\ref{alg:allRHI} for all primes $p \equiv 11 \pmod{12}$
with $10 < p < 660$, implemented in Julia using the \texttt{Oscar}
package~\cite{OSCAR, OSCAR-book} and its dependency
\texttt{Nemo/Hecke}~\cite{nemo}.
The code and data are available at:
\begin{center}
\url{https://github.com/gkorpal/humbert-degree}
\end{center}
Each run was allowed a maximum of two weeks of computation time.
Table~\ref{tab:summary} records the results.
Each row of Table~\ref{tab:summary} corresponds to one prime $p$.
The column $\mathbf{H}$ is the total number of principal polarizations
on $E \times E$ up to equivalence, given by the formula in
\eqref{eq:Hshort}.
The column $\#\Theta$ is the number of polarizations that
Algorithm~\ref{alg:polz} actually found within the time limit, so
$\#\Theta \leq \mathbf{H}$ always, with equality when the computation
finished.
The column $\#\underset{\mathrm{red}}{\mathrm{RHI}}$ counts the
polarizations $\theta \in \Theta$ for which Algorithm~\ref{alg:RHI}
returns a non-zero matrix, that is, the reducible refined Humbert
invariants found.
The column $\mathbf{h}(\mathbf{h}+1)/2$, where $\mathbf{h}$ is given
by~\eqref{h:classnumber}, is the conjectured total count of reducible
refined Humbert invariants over all $\mathbf{H}$ polarizations.
The column $\#\underset{\mathrm{iso}}{\mathrm{RHI}}$ is the number of
isometry classes among those reducible refined Humbert invariants, found
experimentally via Algorithm~\ref{alg:allRHI}.
The column $\#\mathrm{Gen}(q_4)$ is the number of genera of the lattice $q_4$ defined in \eqref{eq:q4},
which equals the true number of isometry classes of
reducible refined Humbert invariants.
Finally, $d_{\mathrm{expt}}$ is the value of $d$ as defined
in~\eqref{maxmin}, computed from the experimental isometry classes in
$\#\underset{\mathrm{iso}}{\mathrm{RHI}}$, and $d_{\mathrm{gen}}$ is
the same quantity computed from the genus classes counted by
$\#\mathrm{Gen}(q_4)$.

When $\#\Theta = \mathbf{H}$, Algorithm~\ref{alg:allRHI} completed
within the time limit and the experimental columns recover the
conjectured values exactly: $\#\underset{\mathrm{red}}{\mathrm{RHI}}$
agrees with $\mathbf{h}(\mathbf{h}+1)/2$,
$\#\underset{\mathrm{iso}}{\mathrm{RHI}}$ agrees with
$\#\mathrm{Gen}(q_4)$, and $d_{\mathrm{expt}}$ agrees with
$d_{\mathrm{gen}}$.
When $\#\Theta < \mathbf{H}$, the computation did not finish within
two weeks, and the experimental columns reflect only the polarizations
found so far.
In this case $\#\underset{\mathrm{red}}{\mathrm{RHI}} \leq
\mathbf{h}(\mathbf{h}+1)/2$,
$\#\underset{\mathrm{iso}}{\mathrm{RHI}} \leq \#\mathrm{Gen}(q_4)$,
and $d_{\mathrm{expt}} \leq d_{\mathrm{gen}}$, as the missing
polarizations may contribute additional invariants and larger degree
values.

\begin{table}[!ht]
    \caption{We get the following results for $\mathsf{allRHI}(p)$ for primes $p\equiv 11\pmod{12}$ with $10 < p < 660$ and a time limit of two weeks. Further explanation in text.}\label{tab:summary}
    \centering 

    \begin{tabular}{ccccccccc}
        \toprule
        $p$ & $\#\Theta$ & $\mathbf{H}$ & $\#\underset{\text{red}}{\text{RHI}}$ & $\frac{\mathbf{h}(\mathbf{h}+1)}{2}$ & $\#\underset{\text{iso}}{\text{RHI}}$ & $\#{\Gen(q_4)}$ & $d_{\text{expt}}$ & $d_{\text{gen}}$  \\
        \midrule 
        11 & 5 & 5 & 3  & 3 & 3 & 3 & 2 & 2\\ 
        23 & 16 & 16 & 6 & 6 & 6 & 6 &  3 & 3\\ 
        47 & 72 & 72 & 15 & 15 & 15 & 15 & 4 & 4\\
        59 & 125 & 125 & 21 & 21 & 21 & 21 &  5 & 5 \\
        71 & 198 & 198 & 28 & 28 & 28 & 28 & 5 & 5\\
        83 & 296 & 296 & 36 & 36 & 29 & 29 & 6 & 6\\  
        107 & 581 & 581 & 55 & 55 & 39 & 39 & 6 & 6\\ 
        131 & 1008 & 1008 & 78 & 78 & 67 & 67 & 7 & 7 \\ 
        167 & 1978 & 1978 & 120 & 120 & 94 & 94 & 8 & 8\\ 
        179 & 2404 & 2404 & 136 & 136 & 97 & 97 & 9 & 9\\ 
        191 & 2886 & 2886 & 153 & 153 & 123 & 123 & 9 & 9 \\ 
        227 & 4708 & 4712 & 206 & 210 & 131 & 135 & 10 & 10 \\ 
        239 & 5458 & 5460 & 229 & 231 & 175 & 177 & 10 & 10 \\ 
        251 & 6275 & 6281 & 247 & 253 & 176 &  181 & 10  & 10\\ 
        263 & 7173 & 7182 & 268 & 276 & 180 & 186 & 10  & 11\\
        311 & 11592 & 11644 & 347 & 378 & 259 & 286 & 11  & 12 \\
        347 & 15770 & 15993 & 412 & 465 & 229 & 265 & 12 & 12 \\
        359 & 17210 & 17654 & 418 & 496 & 279 & 346 & 12  & 12 \\
        383 & 20264 & 21310 & 468 & 561 & 289 & 361  & 12  & 13\\
        419 & 24753 & 27692 & 515 & 666 & 316 & 423 & 13  & 14 \\ 
        431 & 25768 & 30072 & 484 & 703 & 306 & 471 & 13 & 14\\
        443 & 27240 & 32585 & 585 & 741 & 316 & 405 & 12 & 13\\
        467 & 29288 & 38024 & 550 & 820 & 305 & 469 & 13 & 14\\
        479 & 29323 & 40958 & 529 & 861 & 351 & 597 & 13 & 15\\
        491 & 2527 & 44037 & 93 & 903 & 59 & 543 & 10 & 14 \\
        503 & 31476 & 47268 & 564 & 946 & 334 & 594 & 13 & 15 \\
        563 & 33555 & 65808 & 617 & 1176 & 346 & 681 & 14 & 15 \\
        587 & 20151 & 74405 & 413 & 1275 & 229 & 699 & 13 & 16 \\
        599 & 33866 & 78972 & 594 & 1326 & 364 & 832 & 14 & 16\\
        647 & 33331 & 99102 & 595 & 1540 & 347 & 916 & 13 & 16\\
        659 & 33230 & 104620 & 556 & 1596 & 321 & 933 & 14 & 17\\
         \bottomrule
    \end{tabular}
\end{table}

\section{Applications to isogeny problems} \label{sec:apps}
The efficient computation of refined Humbert invariants is an open avenue to work on a variety of isogeny problems\footnote{The potential advantage that we expect from a refined Humbert invariant $q_{(\A,\theta)}$ of a principally polarized abelian surface $(\A,\theta)$ is that it uses the intersection theory of divisors on surfaces, see \cite[\S V.1]{hartshorne2013algebraic} and \cite[\S 4.1]{shafarevich}, and only the degrees of isogenies are used rather than isogenies themselves in the intersection formulas.}. 
They were first used in isogeny-based cryptography in \cite{edachloe}, where their computations were shown to be equivalent to certain isogeny problems.

\subsection{Testing isomorphism of two given principal polarizations}

We now highlight the isomorphism test for principal polarizations used in steps~\ref{alg:polz:20}--\ref{alg:polz:23} of Algorithm~\ref{alg:polz}, which enumerates all principal polarizations on $\A=E\times E$.

For deciding the isomorphism of two principal polarizations $\theta \mapsto (F_{1},F_{2},F_{3},F_{4})$ and 
$\theta' \mapsto (F'_1,F'_2,F'_3,F'_4)$ as in \eqref{traceforms}, we declare that $\theta$ and $\theta'$ are equivalent if their packets are \emph{simultaneously isometric}, i.e., if there exists a matrix
$T\in \mathrm{GL}_8(\mathbb{Z})$
such that $
T F_{i} T^{t}=F'_{i}$ for all  $i=1,2,3,4$.
For computational simplicity, we replace the auxiliary trace forms \eqref{traceforms} by their LLL-reduced forms.
Since the first auxiliary form is positive definite, we first compare inexpensive isometry invariants, including the local genus, initial coefficients of the theta series, and the kissing number ($\kappa$). If any of these invariants differ, we reject the pair; otherwise, we proceed to an exact simultaneous isometry test.

\subsection{Type of a principally polarized superspecial abelian surface} 

The geometric type of a principally polarized superspecial abelian surface can be effectively detected using its refined Humbert invariant and the irreducibility criterion, corresponding exactly to steps~\ref{alg:RHI:5}--\ref{alg:RHI:6} of Algorithm~\ref{alg:RHI}.

More precisely, the irreducibility criterion (Proposition~\ref{irreducibilitycriterion}) shows that $q_{(\A,\theta)}$ represents $1$ if and only if $(\A,\theta)$ is a product of elliptic curves with product polarization; otherwise, it is the Jacobian of a genus-$2$ curve with its canonical principal polarization. This efficient criterion can be incorporated into the $\mathrm{KLPT}^2$ algorithm \cite{CDKLPT25,PRS26} to distinguish the relevant types of principally polarized abelian surfaces\footnote{No direct algorithm is currently known for computing these invariants from curve equations; the criterion therefore applies only when a principal polarization is given as input.}.

\subsection{Improvement on the upper bound on minimum isogeny degrees } 

Recall from \eqref{maxmin} that $d$ is the maximum of the minimum isogeny degree, taken over all pairs $(E_1, E_2)$ of supersingular elliptic curves. Goren--Lauter~\cite{GorenLauterBound} showed that this minimum degree is bounded above by $\frac{2\sqrt{2}}{\pi}\sqrt{p}$, a bound later reproved in~\cite[Lemma 12]{SQIsignHD}. We prove a sharper bound.

\begin{proposition}\label{prop:bound}
Let $E$ and $E'$ be supersingular elliptic curves over $\overline{\mathbb{F}}_p$,
let $\mathcal{A} = E \times E'$, let $q_{E,E'}$ be the degree map as in
Definition~\ref{defn:degreemap}, and suppose the refined Humbert invariant
of $(\mathcal{A},\theta)$ satisfies
$q_{(\mathcal{A},\theta)}(x,\varphi) = x^2 + 4q_{E,E'}(\varphi)$
as in Lemma~\ref{fromRHItodegreemap}. Then $\min q_{E,E'}(\varphi)\leq \sqrt{\frac{p}{2}}$, where the minimum is taken over all nonzero $\varphi \in \Hom(E,E')$. In particular, $d \leq \sqrt{\frac{p}{2}}$.
\end{proposition}

\begin{proof}
Let $L = L_{q_{(\mathcal{A},\theta)}}$ denote the rank-$5$ lattice associated
with $q_{(\mathcal{A},\theta)}$, and let $e = (1,0) \in \mathbb{Z} \oplus \Hom(E,E')$.
Then $q_{(\mathcal{A},\theta)}(e) = 1$, and the decomposition
$q_{(\mathcal{A},\theta)}(x,\varphi) = x^2 + 4q_{E,E'}(\varphi)$ gives an
orthogonal decomposition $L = \mathbb{Z}e \perp M$, where $M = e^\perp \cap L$
has rank $4$ and $q_{(\mathcal{A},\theta)}|_M = 4q_{E,E'}$.
Since $q_{(\mathcal{A},\theta)}$ is an integral quadratic form that represents $1$,
it represents no smaller positive value, so $\lambda_1(L) = 1$.
Every vector of $L$ linearly independent of $e$ has the form $ae + v$ with
$a \in \mathbb{Z}$ and $0 \neq v \in M$, and by orthogonality
$q_{(\mathcal{A},\theta)}(ae + v) = a^2 + q_{(\mathcal{A},\theta)}(v) \geq q_{(\mathcal{A},\theta)}(v)$.

On the other hand, every nonzero vector of $M$ is linearly independent of $e$,
so $\lambda_2(L) = \lambda_1(M)$. Since $q_{(\mathcal{A},\theta)}|_M = 4q_{E,E'}$,
it follows that $\lambda_2(L) = \lambda_1(M) = 4 \min q_{E,E'}(\varphi)$, where the minimum is taken over all nonzero $\varphi \in \Hom(E,E')$.
By \eqref{eq:29p2}, we have $\det(q_{(\mathcal{A},\theta)}) = 2^9 p^2$.
The orthogonal decomposition $L = \mathbb{Z}e \perp M$ with $q_{(\mathcal{A},\theta)}(e) = 1$
then gives $\det(M) =\frac{\det(q_{\A,\theta})}{2}\cdot 2^{-4} = 2^4 p^2$. Applying Minkowski's inequality~\cite[Theorem 2.6.8]{martinet13} in dimension $4$ with Hermite constant $\gamma_4 = \sqrt{2}$~\cite[Table 3.1]{conway2013sphere} gives
\begin{equation*}
  \lambda_1(M) \leq \gamma_4 \det(M)^{1/4} = \sqrt{2} \cdot (4p)^{1/2}
  = 2\sqrt{2p}.
\end{equation*}
Therefore, $4\min q_{E,E'}(\varphi)\leq 2\sqrt{2p}$, and hence
$\min q_{E,E'}(\varphi) \leq \sqrt{\frac{p}{2}}$.
\end{proof}

Our computational experiment verifies the proven bound. For primes $p\equiv 11\pmod{12}$, the maximum, over all pairs of supersingular elliptic curves $E_1,E_2$, of the minimum degree of an isogeny between $E_1$ and $E_2$ is approximately $0.67\sqrt p$.  We observe that  $d < \sqrt{\frac{p}{2}}$, which is illustrated by the plot in Figure~\ref{fig:trendSQI}. 
Moreover, the computations of Goren--Lauter also support that the constant in our bound is close to optimal. Indeed, among the examples reported in \cite[Table~1]{GorenLauterBound}, the largest observed ratio is attained at $p=10007$; this gives
$\frac{70}{\sqrt{10007}} \approx 0.6998$,
which is very close to our theoretical constant
$\frac{1}{\sqrt{2}} \approx 0.7071$.

\begin{figure}[!ht]
    \centering
    \begin{tikzpicture}
    \begin{axis}[
        width=10cm, height=6.5cm,
        xlabel={$p$}, ylabel={$d_{\text{gen}}$},
        ylabel style={rotate=-90},
        legend style={at={(1.02,1)}, anchor=north west, font=\tiny},
        legend cell align=left,
        clip=false,
        grid=major,
        major grid style={line width=0.25pt, draw=gray!25},
        domain=10:700, samples=250
    ]
        \addplot[gray!55, line width=0.7pt, dashed]
            {4.861513*pow(x,0.25) - 8.593322};
        \addlegendentry{$4.8615\,\sqrt[4]{p} - 8.5933$}

        \addplot[gray!55, line width=0.7pt, dotted]
            {2.391002*pow(x,0.3333333333) - 4.490172};
        \addlegendentry{$2.3910\,\sqrt[3]{p} - 4.4902$}

        \addplot[gray!55, line width=0.7pt, densely dashed]
            {4.050280*ln(x) - 11.164465};
        \addlegendentry{$4.0503\,\ln(p) - 11.1645$}

        \addplot[gray!55, line width=0.7pt, loosely dotted]
            {0.021061*x + 4.126784};
        \addlegendentry{$0.0211\,p + 4.1268$}

        \addplot[color=blue!70!black, very thick, solid]
            {0.667944*sqrt(x) - 0.319268};
        \addlegendentry{$0.6679\,\sqrt{p} - 0.3193$}

        \addplot[color=red!70!black, very thick, dashdotted]
            {sqrt(x/2)};
        \addlegendentry{$\sqrt{\frac{p}{2}}\approx 0.7071\sqrt{p}$}

        \addplot[color=black, mark=*, mark size=1.2pt, only marks]
            table [col sep=space, x=p, y=d] {tab.dat};
        \addlegendentry{Data points}
    \end{axis}

    \node[anchor=north west, font=\tiny, align=left, inner sep=0pt]
        at (rel axis cs:1.1,0.45) {%
        \begin{tabular}{@{}lrrr@{}}
            \toprule
            Model & RMSE & MAE & $R^2$ \\
            \midrule
            $\sqrt{p}$ & 0.387 & 0.328 & 0.992 \\
            $\sqrt[3]{p}$ & 0.499 & 0.418 & 0.986 \\
            $\sqrt[4]{p}$ & 0.632 & 0.503 & 0.978 \\
            $p$ & 0.945 & 0.797 & 0.950 \\
            $\ln p$ & 1.161 & 0.949 & 0.925 \\
            \bottomrule
        \end{tabular}
    };
\end{tikzpicture}
    \caption{Visualizing the data collected in Table~\ref{tab:summary} for first 31 primes $p>10$ and $p\equiv 11\mod 12$ and comparing with our bound in Proposition~\ref{prop:bound}. Among the tested two-parameter models, the square-root form gives the best fit. Lower root mean square error (RMSE) and mean absolute error (MAE) are better; higher coefficient of determination ($R^2$) is better.}
    \label{fig:trendSQI}
\end{figure}
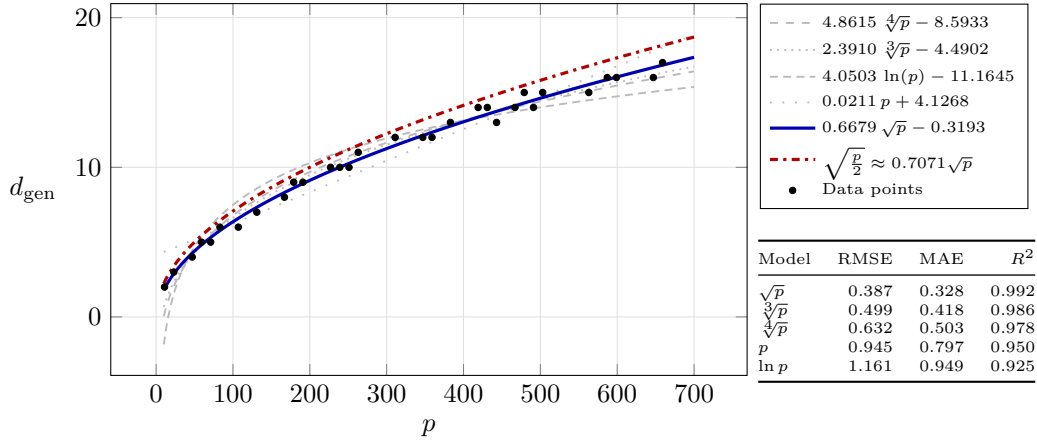

\subsection{The minimum isogeny degree frequency}

The minimum isogeny degree is approximately $\sqrt{p}$, a bound that can be established through several different methods. In particular, Galbraith, Petit, Shani, and Ti showed that an isogeny of degree $\lesssim\sqrt{p}$ is, with high probability, the shortest isogeny between the two curves it connects \cite[\S 4.2]{endomorphismfixedisogeny}. Their argument relies on properties of Ramanujan graphs together with a counting argument, but it does not account for collisions. Love and Boneh, however, show that collisions are rare for degrees $\leq \sqrt{p}$ \cite{loveboneh}.

Our data on degree maps lets us examine how frequently the minimum isogeny degree is achieved, and we find that repetitions are common. For every prime $p \leq 659$ with $p \equiv 11 \pmod{12}$, we computed the complete set of degree maps, and the resulting distributions of minimum isogenies confirm these bounds, as illustrated by the case $p = 659$ in Figure~\ref{fig:degrees}.

\begin{figure}[!ht]
    \centering

\begin{tikzpicture}
    \begin{axis}[
        width=12cm,
        height=6cm,
        ybar interval,
        ymin=0,
        ymax=156,
        enlarge x limits=0.04,
        xlabel={Minimum Degree},
        ylabel={Frequency},
        axis lines=left,
        axis line style={-},
        grid=none,
        xmajorgrids=false,
        ymajorgrids=false,
        tick align=outside,
        xtick=data,
        xticklabels from table={freq.dat}{center},
    ]
        \addplot[
            draw=blue!60!black,
            fill=blue!30,
        ] table [x expr=\thisrow{center}-0.5, y=freq, col sep=space] {freq.dat};

        \addplot[
            only marks,
            mark=none,
            draw=none,
            forget plot,
            nodes near coords,
            every node near coord/.append style={font=\scriptsize, anchor=south, yshift=1pt},
            restrict x to domain=1:17.5,
        ] table [x=center, y=freq, col sep=space] {freq.dat};
    \end{axis}
\end{tikzpicture}
    \caption{For $p=659$, $\mathbf{H}=104,620$ unique polarizations would lead to $\frac{\mathbf{h}(\mathbf{h}+1)}{2}=1596$ reducible refined Humbert invariants belonging to one of the $\#\Gen(q_4)=933$ isometry classes of the form $t_0^2 + 4q_{E_1,E_2}(t_1,t_2,t_3,t_4)$. Here we see the distribution of minimum values of all the degree maps $q_{E_1, E_2}$ we obtained via genus computation.}\label{fig:degrees}
\end{figure}
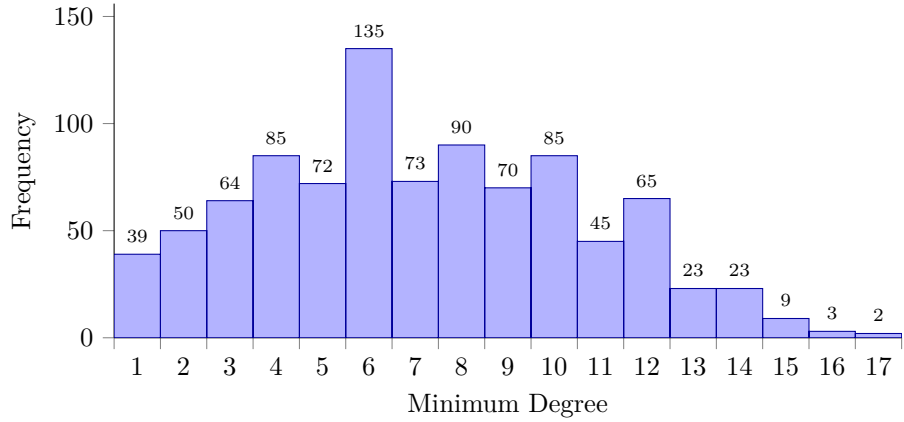

\subsection{The fixed degree isogeny problem}

The \emph{fixed-degree isogeny problem} asks: given supersingular elliptic curves $E_1$ and $E_2$ defined over the finite field $\F_{p^2}$, and given a positive integer $N$, find an isogeny $\varphi: E_1 \rightarrow E_2$ of degree $N$, if it exists. Finding a fixed-degree isogeny is only known to be equivalent to endomorphism ring computation when the isogeny degree is smaller than $\sqrt{p}$, where $p = \charec(K)$ \cite{endomorphismfixedisogeny}. Heuristically, there exists an isogeny of smooth degree $\approx p$, yet no efficient algorithm is known for finding one of that size. For isogenies of degree much larger than $p$, the problem can be solved efficiently using the KLPT algorithm~\cite{KLPT} or its generalization~\cite{SQIsign}, while for small degrees, standard lattice-reduction techniques suffice. The intermediate range of isogeny degrees was investigated in \cite{normform}, which gave improved algorithms for the fixed-degree isogeny problem. Furthermore, \cite{fixeddegreespecialcase} provides a heuristic algorithm that runs in polynomial time for any degree $N$, under certain conditions. The fixed-degree isogeny problem can, in turn, be improved by an efficient algorithm for computing refined Humbert invariants, since this allows us to obtain degree maps.

\section{Conclusion}\label{sec:conclusion}
We summarize our contributions. First, we give an algorithm to enumerate principal polarizations and compute refined Humbert invariants of principally polarized superspecial abelian surfaces. Second, we give an algorithm to determine the type of a given principally polarized superspecial abelian surface. Third, we prove that the upper bound on the minimum isogeny degree between two arbitrary supersingular elliptic curves is $\sqrt{\frac{p}{2}}$ for a generic prime $p$, improving on the bounds of \cite[Section 5.4.1]{GorenLauterBound} and \cite[Lemma 12]{SQIsignHD}, and support this bound experimentally. Fourth, we present experiments on the frequency of minimum isogeny degrees across degree maps over $\overline{\F}_p$ for primes up to $p=659$ with $p\equiv 11\pmod{12}$. Finally, we give a different approach to the fixed-degree isogeny problem by computing degree maps directly from polarization data.

\bibliographystyle{plain}
\bibliography{references}

\end{document}